\documentclass[11pt]{amsart}

\IfFileExists{srcltx.sty}{\usepackage{srcltx}}

\numberwithin{equation}{section}

\usepackage{tikz}
\usetikzlibrary{calc,decorations.pathmorphing}

\usepackage[latin1]{inputenc}
\usepackage{xspace,amssymb,amsfonts,euscript}
\usepackage{amsthm,amsmath}
\usepackage{palatino}
\usepackage{euscript}
\input xy \xyoption {all}

\usepackage{pdfsync}
\usepackage{graphicx}

\RequirePackage{color}
\definecolor{myred}{rgb}{0.75,0,0}
\definecolor{mygreen}{rgb}{0,0.5,0}
\definecolor{myblue}{rgb}{0,0,0.65}

\RequirePackage{ifpdf}
\ifpdf
  \IfFileExists{pdfsync.sty}{\RequirePackage{pdfsync}}{}
  \RequirePackage[pdftex,
   colorlinks = true,
   urlcolor = myblue, 
   citecolor = mygreen, 
   linkcolor = myred, 
   pagebackref,
   bookmarksopen=true]{hyperref}
\else
  \RequirePackage[hypertex]{hyperref}
\fi
\RequirePackage{ae, aecompl, aeguill} 

    \def\CM{{\mathbb{C}}}
    \def\DM{{\mathbb{D}}}

  \def\hg{{\mathfrak h}}

    \def\RM{{\mathbb{R}}}

    \def\ZM{{\mathbb{Z}}}

    \def\BC{{\mathcal{B}}}

    \def\HC{{\mathcal{H}}}

\def\a{\alpha}
\def\b{\beta}
\def\g{\gamma}

\def\e{\varepsilon}

\def\l{\lambda}

\def\z{\zeta}

\newcommand{\nc}{\newcommand} \newcommand{\renc}{\renewcommand}

\newcommand{\rdots}{\mathinner{ \mkern1mu\raise1pt\hbox{.}
    \mkern2mu\raise4pt\hbox{.}
    \mkern2mu\raise7pt\vbox{\kern7pt\hbox{.}}\mkern1mu}}

\def\top{{\mathrm{top}}}
\def\bot{{\mathrm{bot}}}

\DeclareMathOperator{\Tr}{Tr}

\DeclareMathOperator{\Ker}{Ker}

\def\ov{\overline}
\def\un{\underline}

\def\to{\rightarrow}

\def\longto{\longrightarrow}

\def\onto{\twoheadrightarrow}
\nc{\triright}{\stackrel{[1]}{\to}}
\nc{\longtriright}{\stackrel{[1]}{\longto}}

\def\summand{
\mathbin{\tikz [baseline=-3pt]
{
\clip (-0.2,-0.15) rectangle (0.2,0.15);
\node[scale=1.2] at (0,0) {$\subset$}; \node[scale=0.5] at
  (0,0) {$\oplus$}; }}
}

\nc{\ot}{\otimes}

\DeclareMathOperator{\Br}{Br}

\nc{\HotRR}{{}_R\mathcal{K}_R}
\nc{\HotR}{\mathcal{K}_R}
\nc{\excise}[1]{}
\nc{\defect}{\text{df}}
\nc{\h}[1]{\underline{H}_{#1}}

\nc{\ptau}{\tau}

\nc{\Ga}{\mathbb{G}_a} 
\nc{\Gm}{\mathbb{G}_m} 

\nc{\Perv}{{\mathbf{P}}}

\nc{\IH}{{\mathrm{IH}}}

\nc{\ic}{\mathbf{IC}}

\nc{\gl}{{\mathfrak{gl}}}
\renc{\sl}{{\mathfrak{sl}}}
\renc{\sp}{{\mathfrak{sp}}}

\nc{\HBM}{H^{BM}}
\nc{\id}{\textrm{id}}

 \DeclareMathOperator{\Hom}{Hom}

\DeclareMathOperator{\supp}{supp} \DeclareMathOperator{\ch}{ch}
\DeclareMathOperator{\End}{End} 

\DeclareMathOperator{\rad}{rad}

\newtheorem{thm}{Theorem}[section]
\newtheorem{lem}[thm]{Lemma}

\newtheorem{prop}[thm]{Proposition}
\newtheorem{cor}[thm]{Corollary}

\theoremstyle{definition}

\theoremstyle{remark}
\newtheorem{remark}[thm]{Remark}

\newcommand{\into}{\hookrightarrow}

\DeclareMathOperator{\Gr}{Gr}
\nc{\simto}{\stackrel{\sim}{\to}}
\nc{\oBC}{\ov{\BC}} 

\DeclareMathOperator{\rk}{rk}
\DeclareMathOperator{\gdim}{\un{\dim}}
\DeclareMathOperator{\grk}{\un{\rk}}

\begin{document}

\title{The Hodge theory of Soergel bimodules}

\author{Ben Elias}
\address{Massachusetts Institute of Technology, Boston, USA}
\email{belias@math.mit.edu}
\urladdr{}

\author{Geordie Williamson} 
\address{Max-Planck-Institut f\"ur Mathematik, Bonn, Germany}
\email{geordie@mpim-bonn.mpg.de}
\urladdr{}

\begin{abstract} We prove Soergel's conjecture on the characters of indecomposable Soergel bimodules. We deduce that Kazhdan-Lusztig polynomials have
positive coefficients for arbitrary Coxeter systems. Using results of
Soergel one may deduce an algebraic proof of the Kazhdan-Lusztig conjecture. \end{abstract}

\maketitle

\tableofcontents

\section{Introduction} \label{sec:introduction}

In 1979 Kazhdan and Lusztig introduced the Kazhdan-Lusztig basis of the Hecke algebra of a Coxeter system \cite{KL1}. The definition of the Kazhdan-Lusztig basis is elementary, however it
appears to enjoy remarkable positivity properties. For example, it is conjectured in \cite{KL1} that Kazhdan-Lusztig polynomials (which express the Kazhdan-Lusztig basis in terms of the
standard basis of the Hecke algebra) have positive coefficients. The same paper also proposed the Kazhdan-Lusztig conjecture, a character formula for simple highest weight modules for a
complex semi-simple Lie algebra in terms of Kazhdan-Lusztig polynomials associated to its Weyl group.

In a sequel \cite{KL2}, Kazhdan and Lusztig established that their polynomials give the Poincar\'e polynomials of the local intersection cohomology of Schubert varieties (using Deligne's
theory of weights), thus establishing their positivity conjectures for
finite and affine Weyl groups. In 1981 Beilinson and Bernstein
\cite{BB} and Brylinski and Kashiwara \cite{BK} established a 
connection between highest weight representation theory and perverse
sheaves, using $D$-modules and the Riemann-Hilbert
correspondence, thus proving the
Kazhdan-Lusztig conjecture. Since their introduction
Kazhdan-Lusztig polynomials have become ubiquitous throughout highest
weight representation theory, giving character formulae for affine Lie
algebras, quantum groups at a root of unity, rational representations
of algebraic groups, etc.

In 1990 Soergel \cite{S1} gave an alternate proof of the
Kazhdan-Lusztig conjecture, using certain modules over the cohomology
ring of the flag variety.\footnote{\cite[\S1.1, Bermerkung
  5]{S1}. This seems not to be as well-known as it should be.} In a subsequent
paper \cite{S2} Soergel introduced equivariant analogues of these modules, which have come to be known as \emph{Soergel bimodules}.

Soergel's approach is remarkable in its simplicity. Using only the action of the Weyl group on a Cartan subalgebra, Soergel associates to each simple reflection a graded bimodule over the
regular functions on the Cartan subalgebra. He then proves that the split Grothendieck
group of the monoidal category generated by these 
bimodules (the category of Soergel bimodules) is isomorphic to the
Hecke algebra. Moreover, the Kazhdan-Lusztig conjectures (as well as
several positivity conjectures) are equivalent to 
the existence of certain bimodules whose classes in the Grothendieck
group coincide with the Kazhdan-Lusztig basis. Despite its elementary
appearance, this statement is difficult to 
verify. For finite Weyl groups, Soergel deduces the existence of such bimodules by
applying the decomposition theorem of Beilinson, Bernstein, Deligne
and Gabber \cite{BBD} to 
identify the indecomposable Soergel bimodules with the equivariant
intersection cohomology of Schubert varieties. This approach was
extended by H\"arterich to the setting of Weyl groups of
symmetrizable Kac-Moody groups \cite{Ha}.  Except for his appeal
to the decomposition theorem, Soergel's approach is entirely
algebraic. (The decomposition theorem relies on the base field having characteristic 0, which will be an important assumption below.)

In \cite{S2} and \cite{S3} Soergel pointed out that the algebraic
theory of Soergel bimodules can be developed for an arbitrary Coxeter
system. Starting with an 
appropriate representation of the Coxeter group (the substitute for
the Weyl group's action on a Cartan subalgebra) one defines the
monoidal category of 
Soergel bimodules by mimicking the Weyl group case. Surprisingly, one
again obtains a monoidal category whose split Grothendieck group is
canonically identified with the Hecke algebra. 
Soergel then conjectures the existence (over a field of characteristic
0) of indecomposable bimodules whose classes coincide with the
Kazhdan-Lusztig basis of the Hecke algebra. At this level of
generality there is no known recourse to geometry. One does not
have a flag variety or Schubert varieties associated to arbitrary
Coxeter groups, and so one has no geometric setting in which to apply
the decomposition theorem. Soergel's conjecture was established for
dihedral groups by Soergel \cite{S2} and for ``universal'' Coxeter systems
(where each product of simple reflections has infinite order) by
Fiebig \cite{Funiv} and Libedinsky. However, in both these cases there
already existed closed formulas for the Kazhdan-Lusztig polynomials.

In this paper we prove Soergel's conjecture for an arbitrary Coxeter
system. We thus obtain a proof of the
positivity of Kazhdan-Lusztig polynomials (as well as several other
positivity conjectures). We also obtain an algebraic proof of the
Kazhdan-Lusztig conjecture, completing the program initiated by
Soergel. In some sense we have come full circle: the original paper of Kazhdan and Lusztig was stated in the generality of an arbitrary Coxeter system, this paper returns Kazhdan-Lusztig theory to this
level of generality.

Our proof is inspired by two papers of de Cataldo and Migliorini
(\cite{dCM1} and \cite{dCM2}) which give Hodge-theoretic proofs of the
decomposition theorem.  In essence, de Cataldo and Migliorini show that the decomposition theorem for a proper map (from a smooth space) is implied by certain Hodge theoretic properties of the cohomology groups of the source, under a Lefschetz operator induced from the target.
We discuss their approach in more 
detail below. Thus they are able to transform a geometric question on the
target into an algebraic question on the source. They then 
use classical Hodge theory and some ingenious arguments to complete
the proof. For Weyl groups, 
Soergel bimodules are the equivariant intersection
cohomology of Schubert varieties, and as such have a
number of remarkable Hodge-theoretic 
properties which seem not to have been made explicit before.
In fact, these properties hold for any Coxeter group;
Soergel bimodules always behave as though they were intersection
cohomology spaces of projective varieties! In this 
paper, we give an algebraic proof of these Hodge-theoretic properties,
for any Coxeter group, and adapt the proof that these Hodge-theoretic
properties imply the ``decomposition 
theorem", at least insofar as Soergel's conjecture is concerned.

Here are some highlights of de Cataldo and Migliorini's proof from \cite{dCM1}:
\begin{enumerate}
\item ``Local intersection forms'' (which control the decomposition of
  the direct image of
  the constant sheaf) can be embedded into ``global intersection
  forms" on the cohomology of smooth varieties.
\item The Hodge-Riemann bilinear relations can be used to conclude that the restriction of a form to a subspace (i.e. the image of a local intersection form) stays definite.
\item One should first prove the hard Lefschetz theorem, and then
  deduce the Hodge-Riemann bilinear relations via a limiting argument
  from a family of known cases, using that the signature of a non-degenerate symmetric real form cannot change in a family.
\end{enumerate}

It is this outline that we adapt to our algebraic situation. However the translation of their results into the language of Soergel bimodules is by no means automatic. The
biggest obstacle is to find a replacement for the use of hyperplane sections and the weak Lefschetz theorem. We believe that our use of the Rouquier complex to overcome this
difficulty is an important observation and may have other applications.

There already exists a formidable collection of algebraic machinery,
developed by Soergel \cite{S2,S4}, Andersen-Jantzen-Soergel
\cite{AJS}, and Fiebig \cite{F1,F2}, which
provides algebraic proofs of many 
deep results in representation theory once Soergel's conjecture is
known.
These include the Kazhdan-Lusztig conjecture for affine Lie algebras
(in non-critical level), the Lusztig conjecture for quantum groups at
a root of unity, and the Lusztig 
conjecture on modular characters of reductive algebraic groups in
characteristic $p \gg 0$.

There are many formal similarities between the
theory we develop here, and the theory of intersection cohomology of
non-rational polytopes, which was developed to prove Stanley's
conjecture on the unimodularity of the generalized $h$-vector
\cite{BL, K, BBFK}. In both cases one obtains spaces which look like the 
intersection cohomology of a (in many cases non-existent) projective
algebraic variety. Dyer \cite{D1,D2} has a proposed a conjectural framework for
understanding both of these theories in parallel. It would be 
interesting to know whether the techniques of this paper shed light on
this more general theory.

\subsection{Results} \label{sec:results}

Fix a Coxeter system $(W,S)$. Let $\HC$ denote the Hecke algebra of
$(W,S)$, a $\ZM[v^{\pm 1}]$-algebra with standard basis $\{H_x\}_{x \in W}$ and Kazhdan-Lusztig
basis $\{\un{H}_x\}_{x \in W}$ as in \S\ref{sec:hecke}.
We fix a reflection faithful (in the sense of \cite[Definition 1.5]{S3}) representation $\hg$ of $W$ over $\RM$ and let
$R$ denote the regular functions on $\hg$, graded with $\deg \hg^* =2$. We denote by $\BC$ the category of Soergel bimodules; it is the full
additive monoidal Karoubian subcategory of graded $R$-bimodules generated by $B_s := R \otimes_{R^s} R(1)$ for all $s \in S$. Here, $R^s \subset R$
denotes the subalgebra of $s$-invariants, and $(1)$ denotes the grading shift which places the element $1 \otimes 1$ in degree $-1$. For any $x$ there exists up to isomorphism a unique indecomposable Soergel bimodule $B_x$ which occurs as
a direct summand of the Bott-Samelson bimodule $BS(\un{x}) = B_s\otimes_R B_t \otimes_R \dots \otimes_R B_u$ for any reduced expression $\un{x} =
st\dots u$ for $x$, but does not occur as a summand of any Bott-Samelson bimodule for a shorter expression. The bimodules $B_x$ for $x \in W$ give
representatives for the isomorphism classes of all indecomposable Soergel bimodules up to shifts. The split Grothendieck group $[\BC]$ of the
category of Soergel bimodules is isomorphic to $\HC$. The character $\ch(B) \in \HC$ of a Soergel bimodule $B$ is a $\ZM_{\ge 0}[v^{\pm}]$-linear
combination of standard basis elements $\{H_x\}$ given by counting ranks of subquotients in a certain canonical filtration; it realizes the class of
$B$ under the isomorphism $[\BC] \simto \HC$.

\begin{thm} (Soergel's conjecture) \label{thm:SC} For all $x \in W$ we
  have $\ch(B_x) = \un{H}_x$. \end{thm}

Because $\ch(B)$ is manifestly positive we obtain:

\begin{cor} (Kazhdan-Lusztig positivity conjecture) \begin{enumerate} \item If we write $\un{H}_x = \sum_{y \le x} h_{y,x} H_y$ then $h_{y,x} \in \ZM_{\ge 0}[v]$. \item If we write
$\un{H}_x\un{H}_y = \sum \mu^z_{x,y} \un{H}_z$ then $\mu^z_{x,y} \in \ZM_{\ge 0}[v^{\pm 1}]$. \end{enumerate} \end{cor}
(See Remark \ref{rem:oldKL} for the relation between our notation and
that of \cite{KL1}.)

We prove that indecomposable Soergel bimodules have all of the algebraic properties
known for intersection cohomology.  Given a Soergel bimodule $B$, we
denote by $\ov{B} := B \ot_R \RM$ the quotient by the
image of positive degree polynomials acting on the right.
We let
$(\overline{B})^i$ denote the degree $i$ component of
$\overline{B}$. The self-duality of Soergel bimodules implies that
$\dim_{\RM}(\overline{B_x})^{-i} = \dim_{\RM}(\overline{B_x})^i$ for all $i$.
For the rest of the paper we fix a degree two element $\rho \in \hg^*$
which is strictly positive on any simple coroot $\alpha_s^\vee \in
\hg$ (see \S\ref{sec:cox}).

\begin{thm} (Hard Lefschetz for Soergel bimodules) \label{thm:HL} The action of $\rho$ on $B_x$ by left multiplication induces an operator on $\ov{B_x}$ which satisfies the hard
Lefschetz theorem. That is, left multiplication by $\rho^i$ induces an
isomorphism \[ \rho^i  : (\ov{B_x})^{-i} \simto (\ov{B_x})^i. \] \end{thm}

We say that a graded $R$-valued form \[ \langle -, - \rangle : B_x \times B_x \to R \] is \emph{invariant} if it is bilinear for the right action of $R$, and if $\langle rb,
b' \rangle = \langle b, rb' \rangle$ for all $b,b' \in B$ and $r \in
R$. Theorem \ref{thm:SC} and Soergel's hom formula (see Theorem
\ref{thm:hom}) imply that the
degree zero endomorphisms of $B_x$ consist only of scalars, i.e. $\End(B_x) = \RM$. Combining this with the self-duality of
indecomposable Soergel bimodules, we see that there exists an
invariant form  $\langle -, - \rangle_{B_x}$ on $B_x$ which is unique up to a
scalar. We write $\langle -, - \rangle_{\ov{B_x}}$ for the $\RM$-valued form on $\ov{B_x}$ induced by $\langle -, - \rangle_{B_x}$. We
fix the sign on $\langle -, - \rangle_{B_x}$ by requiring that $\langle \ov{c}, \rho^{\ell(x)} \ov{c} \rangle_{\ov{B_x}} > 0$, where $c$ is any generator of $B_x^{-\ell(x)} \cong \RM$.
With this additional positivity constraint, we call $\langle -, -
\rangle_{B_x}$ the \emph{intersection form} on $B_x$. It is well-defined up to positive scalar.

\begin{thm} (Hodge-Riemann bilinear relations) \label{thm:HR} For all
  $i \ge 0$ the Lefschetz form on $(\ov{B_x})^{-i}$ defined by \[ (\alpha, \beta)_{-i}^\rho := \langle \a, \rho^i \b \rangle_{\ov{B_x}} \]
is $(-1)^{(-\ell(x) + i)/2}$-definite when restricted to the primitive
subspace \[P_\rho^{-i} = \ker( \rho^{i+1}) \subset (\ov{B_x})^{-i}.\] \end{thm}

Note that $B_x^{-i} = 0$ unless $i$ and $\ell(x)$ are congruent modulo
2. Throughout this paper we adopt the convention that if $m$ is odd then a space is $(-1)^{\frac{m}{2}}$-definite
if and only it is zero. The reader need not worry too much about the sign in this and other Hodge-Riemann statements. Throughout the introduction the form on the
lowest non-zero degree will be positive definite, and the signs on
primitive subspaces will alternate from there upwards.

As an example of our results, consider the case when $W$ is
finite.  If $w_0 \in W$
denotes the longest element of $W$, then $B_{w_0} = R \otimes_{R^W} R(\ell(w_0))$, where $R^W$
denotes the subalgebra of $W$-invariants in $R$. Hence
\[
\overline{B_{w_0}} = (R \otimes_{R^W} R) \otimes_{R} \RM (\ell(w_0)) = R/( (R^W)^+)(\ell(w_0))
\]
is the coinvariant ring, shifted so as to have Betti numbers
symmetric about zero (here $((R^W)^+)$ denotes the ideal of $R$
generated by elements of $R^W$ of positive degree). The coinvariant ring
is equipped with a canonical symmetric form and Theorems
\ref{thm:HL} and \ref{thm:HR} yield that left multiplication by any
$\rho$ in the interior of the dominant chamber of $\hg^*$ satisfies
the hard Lefschetz theorem and Hodge-Riemann bilinear relations.

If $W$ is a Weyl group of a compact Lie group $G$, then the coinvariant
ring above  is isomorphic to the real cohomology ring of the flag
variety of $G$ and the hard Lefschetz theorem and Hodge-Riemann
bilinear relations follow from classical Hodge theory, because the
flag variety is a projective algebraic variety. On the other hand if
$W$ is not a Weyl group (e.g. a non-crystallographic dihedral group,
or a group of type $H_3$ or $H_4$) then there is no obvious geometric
reason why the hard Lefschetz theorem or Hodge-Riemann bilinear
relations should hold. The hard Lefschetz property for coinvariant
rings has been studied by
a number of authors \cite{MNW, MW, McD} but even for the coinvariant rings of
$H_3$ and $H_4$ the fact that the Hodge-Riemann bilinear relations
hold seems to be new.

\subsection{Outline of the proof} \label{sec:outline}

\subsubsection{Setup}
Our proof is by induction on the Bruhat order, and the hard Lefschetz
property and Hodge-Riemann bilinear relations play an essential role
along the way. Throughout this paper we employ the following
abbreviations for any $x \in W$:
\begin{gather*} S(x) : \begin{array}{c}
\text{Soergel's conjecture for $B_x$: }\\
\text{Theorem \ref{thm:SC} holds for $x$.}\end{array} \\
hL(x): \begin{array}{c} \text{hard Lefschetz for
      $\ov{B_x}$:}\\
\text{Theorem \ref{thm:HL} holds for $x$.} \end{array} \\
HR(x): \begin{array}{c} \text{the Hodge-Riemann bilinear
      relations for $\ov{B_x}$:}\\\text{$S(x)$ holds and Theorem \ref{thm:HR}
      holds for $x$.} \end{array}\end{gather*}
The abbreviation $hL(<\!x)$
means that $hL(y)$ holds for all $y < x$. Similar interpretations hold
for abbreviations like $S(\le\!x)$, etc.

In the statement of $HR(x)$ it is necessary to assume $S(x)$ to ensure the
uniqueness (up to positive scalar) of the intersection form on $B_x$. 
However we need not assume $S(x)$ in order to ask
whether a given form on $B_x$ (not necessarily the intersection form)
induces a form on $\overline{B_x}$ satisfying the Hodge-Riemann
bilinear relations. Now $B_x$ appears as a summand of the
Bott-Samelson  
bimodule $BS(\un{x})$ for any reduced expression $\un{x}$ for $x$. 
Bott-Samelson bimodules are equipped with an explicit symmetric non-degenerate
\emph{intersection form} defined using the ring
structure and a trace on $BS(\un{x})$ (just as the
intersection form on the cohomology of a smooth projective variety is
given by evaluating the fundamental class on a product). The following
stronger version of $HR(x)$ is more useful in induction
steps, as it can be posed without assuming $S(x)$:
\begin{gather*}
HR(\un{x}): \begin{array}{c} \text{for any
    embedding $B_x \subset
    BS(\un{x})$} \\
\text{the Hodge-Riemann bilinear
      relations hold:}\\\text{the conclusions of Theorem \ref{thm:HR}
      hold for the } \\
\text{restriction of the intersection form on
      $BS(\un{x})$ to $B_x$.} \end{array}
\end{gather*}
(Here and elsewhere an ``embedding'' of Soergel
bimodules means an ``embedding as a direct summand''.)
Together, $S(x)$ and $HR(\un{x})$ imply that the restriction of the
intersection form on $BS(\un{x})$ to $B_x$ agrees with the intersection
form on $B_x$ up to a positive scalar, for any choice of embedding
(see Lemma \ref{lem:HRembedding} for the proof). In other words:
\begin{equation}
  \label{eq:1}\begin{matrix}
  \text{If $S(x)$ holds, then $HR(x)$ and $HR(\un{x})$ are equivalent,}\\
 \text{for any reduced expression $\un{x}$ of $x$.}
\end{matrix}
\end{equation}

We now give the structure of the proof. In \S \ref{ss:sclif}, \S
\ref{ss:lgic} and \S \ref{sc:la} we introduce and explain the
implications between statements needed to perform the induction. In \S\ref{ss:sp} we
give a summary of the induction.

We make the following assumption:
\begin{equation}
\begin{matrix} \text{In \S \ref{ss:sclif}, \S
\ref{ss:lgic} and \S \ref{sc:la} we fix $x \in W$ and $s \in S$ with
$xs > x$} \\
\text{and assume that $S(<\!xs)$ holds.}
\end{matrix}
\end{equation}
\subsubsection{Soergel's conjecture and the local intersection form} \label{ss:sclif}
By Soergel's hom formula (see Theorem \ref{thm:hom}), $S(<\!xs)$ is
equivalent to assuming that $\End(B_y) = \RM$ for all $y <
xs$. Consider the 
form given by composition \[ (-,-)_y^{x,s}: \Hom(B_y, B_xB_s) \times
\Hom(B_xB_s, B_y) \to \End(B_y) = \RM. \] Soergel's hom formula gives
an expression for the dimension of these hom spaces in terms of an
inner product on the Hecke algebra.  Applying this formula one sees that $S(xs)$ is
equivalent to the non-degeneracy of this form for all $y < xs$ (see
\cite[Lemma 7.1(2)]{S3}). Now $B_y$ and $B_xB_s$ are naturally
equipped with symmetric invariant bilinear forms (see \S \ref{sec:embedding}) so there is a
canonical identification (``take adjoints'') \[ \Hom(B_y, B_xB_s) =
\Hom(B_xB_s, B_y). \] Hence  we can view $(-,-)_y^{x,s}$ as a bilinear form on the real vector space
$\Hom(B_y, B_xB_s)$. We call this form the \emph{local intersection form}. We consider ``Soergel's conjecture with signs'': \begin{gather*} S_{\pm}(y,x,s): \text{the form $(-,-)_y^{x,s}$ is
$(-1)^{(\ell(x)+1-\ell(y))/2}$-definite.} \end{gather*} This is a
priori stronger than Soergel's conjecture. By the above discussion:
\begin{equation}
  \label{eq:2}
  \text{$S(<\!xs)$ and $S_{\pm}(<\!xs,x,s)$ imply $S(xs)$.}
\end{equation}

\subsubsection{From the local to the global intersection form} \label{ss:lgic}
To prove $S_{\pm}(y,x,s)$, we must digress and discuss hard
Lefschetz and the Hodge-Riemann bilinear relations for 
$\overline{B_xB_s}$. The connection is explained by \eqref{eq:3} below.
Recall that we have fixed a degree two element $\rho \in R$ such that
$\rho(\a^\vee_s) > 0$ for all simple coroots $\a_s^\vee$. Consider the
``hard Lefschetz'' condition: \begin{gather*} hL(x,s): \quad 
  \text{$\rho^i : (\ov{B_xB_s})^{-i} \to (\ov{B_xB_s})^i$ is an
isomorphism.} \end{gather*} Because $B_{xs}$ is a direct summand of
$B_xB_s$, $hL(x,s)$ implies $hL(xs)$. They are equivalent if we know
$hL(<\!xs)$, since every other indecomposable summand of
$B_x B_s$ is of the form $B_y$ for $y<xs$ (a consequence of
our standing assumption $S(<\!xs)$).

If we fix a reduced expression $\un{x}$ for $x$ and an embedding $B_x
\subset BS(\un{x})$ then $B_x$ inherits an
invariant form from $BS(\un{x})$ as discussed above.  Similarly, $B_x B_s$ is a summand
of $BS(\un{x}s)$ and inherits an invariant form, which we denote
$\langle -, - \rangle_{B_x B_s}$. We formulate the Hodge-Riemann bilinear relations for
$\ov{B_xB_s}$ as follows:
\begin{gather*} HR(\un{x},s): \begin{array}{c}
\text{for any embedding $B_x \subset BS(\un{x})$}\\
\text{the Lefschetz form $(\a, \b)^{-i}_\rho
:= \langle \a, \rho^i \b \rangle_{\ov{B_x B_s}}$ is}\\ \text{$(-1)^{(\ell(x) + 1 - i)/2}$-definite on the primitive subspace } \\ \text{} P^{-i}_{\rho} := \ker (\rho^{i+1}) \subset
(\ov{B_xB_s})^{-i}. \end{array} \end{gather*}

Once again, using that $B_x B_s \cong B_{xs} \oplus \bigoplus B_y^{\oplus m_y}$ for some $m_y \in \ZM_{\ge 0}$ one may deduce easily that
$HR(\un{x},s)$ implies $HR(\un{x}s)$ (see Lemma \ref{lem: HR subspace}). However $HR(\un{x},s)$ is stronger than assuming $HR(\un{x}s)$ and
$HR(y)$ for all $y < xs$ with $m_y \ne 0$, because it fixes the sign of the restricted form. Indeed, $HR(\un{x},s)$ is equivalent to the statement
that the restriction of $\langle -, - \rangle_{B_xB_s}$ to any summand $B_y$ of $B_x B_s$ is $(-1)^{(\ell(xs)-\ell(y))/2}$ times a positive multiple
of the intersection form on $B_y$. For later use, we employ the following abbreviation:
\begin{gather*} HR(x,s): \begin{array}{c}
\text{$HR(\un{x},s)$ holds, for all reduced expressions $\un{x}$ of $x$.}
 \end{array} \end{gather*}

Recall that the space $\Hom(B_y, B_xB_s)$ is equipped with the local
intersection form $(-,-)_y^{x,s}$ and that $(\overline{B_x
  B_s})^{-\ell(y)}$ is equipped with the Lefschetz form
  $(-,-)_\rho^{-\ell(y)}$. The motivation for introducing the
  condition $HR(\un{x},s)$ is the following (see Theorem \ref{thm:embedding}):
for any $\rho$ as above there exists an embedding: \[ \iota: \Hom(B_y, B_xB_s) \hookrightarrow P^{-\ell(y)}_\rho \subset
(\overline{B_x B_s})^{-\ell(y)}. \] Moreover, this embedding is an isometry up to a positive scalar.

Because the restriction of a definite form to a subspace is definite,
we obtain:
\begin{equation}
  \label{eq:3}
  \text{$S(<\!xs)$ and $HR(\un{x},s)$ imply $S_{\pm}(<\!xs,x,s)$.}
\end{equation}
Combining \eqref{eq:3} and \eqref{eq:2} and the above discussion, we
arrive at the core statement of our induction:
\begin{equation}
  \label{eq:core}
  \text{$S(<\!xs)$ and $HR(\un{x},s)$ imply $S(\le \!xs)$ and $HR(\un{x} s)$.}
\end{equation}
It remains to show that $S(\le\!x)$ and $HR(\le\!x)$ implies $HR(\un{x},s)$. This reduces Soergel's conjecture to a  statement about the
modules $\ov{B_xB_s}$ and their intersection forms. 

\subsubsection{Deforming the Lefschetz operator} \label{sc:la}
The reader might have noticed that $hL$ seems to have disappeared from the
picture. Indeed, $HR$ is stronger than $hL$, and one might ask why we
wish to treat $hL$ separately. The reason is that it seems extremely
difficult to attack $HR(\un{x}, s)$ directly.  As we noted earlier, de Cataldo and Migliorini's
method of proving $HR$ consists in proving $hL$ first for a family of operators, and using a limiting argument to deduce $HR$.

We adapt their limiting argument as follows. For any real number $\z \ge 0$, consider the Lefschetz operator \[ L_\z := (\rho \cdot -) \id_{B_s} + \id_{B_x} (\z \rho \cdot-) \] which we view as an
endomorphism of $B_xB_s$. Here $(\rho \cdot - )$ (resp. $ (\zeta \rho \cdot -)$) denotes the operator of left multiplication on $B_x$ (resp. $B_s$) by $\rho$ (resp. $\z \rho$) and
juxtaposition denotes tensor product of operators. Now consider the
following ``$\z$-deformations'' of the above
statements:
\begin{gather*} hL(x,s)_\z: \quad \text{$L_\z^i : (\ov{B_xB_s})^{-i} \to
(\ov{B_xB_s})^i$ is an isomorphism.} \end{gather*} \begin{gather*}
HR(\un{x},s)_\z: \quad \begin{array}{c}  \text{for any embedding $B_x
    \subset BS(\un{x})$} \\
\text{the Lefschetz form $(\a, \b)_{-i}^\rho := \langle \a, L_\z^i \b \rangle_{\ov{B_x
B_s}}$}\\ \text{is $(-1)^{(\ell(x) + 1 - i)/2}$-definite on the
primitive subspace } \\ \text{} P^{-i}_{L_\z} := \ker (L_{\z}^i)
\subset (\ov{B_xB_s})^{-i}. \end{array}\\
  HR(x,s)_\z: \begin{array}{c}
\text{$HR(\un{x},s)_\z$ holds, for all reduced expressions $\un{x}$ of
  $x$.} \end{array}
\end{gather*}

Note that $L_0$ is simply left multiplication by $\rho$, and hence
$hL(x,s)_0 = hL(x,s)$, $HR(\un{x},s)_0 = HR(\un{x},s)$ and $HR(x,s)_0
= HR(x,s)$. The signature of a family of non-degenerate symmetric real
forms cannot change in the family. Therefore, if $hL(x,s)_\z$ holds for all $\z \ge 0$ and $HR(\un{x},s)_\z$ holds for any single non-negative value of $\z$, then $HR(\un{x},s)_0$ also holds. (This is the essence
of de Cataldo and Migliorini's limiting argument.)

The first hint that this deformation is promising is Theorem \ref{thm:towards infinity}:
\begin{equation}
  \label{eq:limit}
  \text{$HR(\un{z})$ implies $HR(\un{z}, s)_\z$ for $\z \gg 0$}
\end{equation}
(which holds regardless of whether $zs > z$ or $zs < z$). Therefore, we have
\begin{equation}
  \label{eq:hLzs>s}
  \text{$HR(\un{x})$ and $hL(x,s)_\z$ for all $\z \ge 0$, implies $HR(\un{x},
    s)_\z$ for all $\z \ge 0$.}
\end{equation}
In particular, the fact that $hL(z,s)_\z$ and $HR(\un{z},s)_\z$ hold for all $\z \ge 0$ and all $z<x$ with $sz > z$ is something we may inductively assume, when trying to prove the same facts for $x$.

We have reduced our problem to establishing $hL(x,s)_\z$ for $\z \ge 0$. In de Cataldo and Migliorini's approach this is established using the weak
Lefschetz theorem and the Hodge-Riemann bilinear relations in smaller dimension. In our setting the weak Lefschetz theorem is missing, and a key
point is the use of Rouquier complexes as a replacement (see the first few paragraphs of \S\ref{sec:HL} for more details). The usual proof of $hL$
for a vector space $V$ is to find a map $V \to W$ of degree 1, injective on $V^{-i}$ for $i>0$ and commuting with the Lefschetz operator, where $HR$
is known to hold for $W$. The Rouquier complex yields a map of degree 1 from $B_x B_s$, injective on negative degrees and commuting with $L$, to a
direct sum of $B_x$ and terms of the form $B_z B_s$ for summands $B_z$
of $BS(\un{x})$ with $z < x$. This target space does not satisfy the
Hodge-Riemann bilinear relations, but
nevertheless we are able to prove the hard Lefschetz theorem.

When $\z=0$, we have an argument which shows:
\begin{equation}\label{eq:hLz=0} \begin{array}{c} \text{$S(\le\!x)$,
      $hL(<\!xs)$, $HR(x)$ and $HR(z,t)$ for all $z < x$ with $zt > z$}\\
\text{ together imply $hL(x,s)$.} \end{array} \end{equation}
This is Theorem \ref{thm:hL}. One feature of the proof is that, whenever $zs < z$, the decomposition $B_z B_s \cong B_z(1) \oplus B_z(-1)$
commutes with the Lefschetz operator $L_0$. This decomposition allows one to bypass the fact that $HR(z,s)$ fails if $zs < z$.

When $\z>0$, the decomposition $B_z B_s \cong B_z(1) \oplus B_z(-1)$ for $zs < z$ does not commute with $L_\z$. However, proving $hL(z,s)_\z$ for
$\z>0$ and $zs < z$ using $hL(z)$ is a straightforward computation (Theorem \ref{thm:hLI}). Our inductive hypotheses and the limiting argument above
now yield $HR(\un{z},s)_\z$ for all $z<x$. A similar argument to the previous case shows: \begin{equation}\label{eq:hLz>0} \begin{array}{c}\text{For $\z>0$,
$S(\le\!x)$, $HR(\le\!x)$, $HR(<\! x,s)_\z$}\\\text{and $HR(z,t)$ for
all $z < x$ with $zt > z$ imply $hL(x,s)_\z$.} \end{array} \end{equation} This is Theorem \ref{thm:hLII}.

\subsubsection{Structure of the proof} \label{ss:sp}
Let us summarise the overall inductive proof. Let $X \subset W$ be an
ideal in the Bruhat order (i.e. $z \le x \in X \Rightarrow z \in X$) and assume:
\begin{enumerate}
\item $HR(z,t)_\z$ for all $\z \ge 0$, $z < zt \in X$ and $t \in S$,
\item $HR(z,t)_\z$ for all $\z > 0$, $zt < z \in X$ and $t \in S$.
\end{enumerate}
We have already explained why (1) implies $S(X)$, $hL(X)$ and $HR(X)$.

Now choose a minimal element $x'$ in the complement of $X$, and choose $s \in S$ and $x \in X$ with $x' = xs$. As we just discussed,
\eqref{eq:hLz=0} and \eqref{eq:hLz>0} imply that $hL(x,s)_\z$ holds for all $\z \ge 0$. Using $HR(\un{x})$ and \eqref{eq:limit} we deduce
$HR(x,s)_\z$ for all $\z \ge 0$. Therefore (1) holds with $X$ replaced by $X \cup \{x'\}$, and thus $S(x')$, $hL(x')$, and $HR(x')$ all hold.

As above, the straightforward calculations of Theorem \ref{thm:hLI} show that $hL(x',t)_\z$ holds for $\z>0$ when $t \in S$ satisfies $x't < x'$.
Again by $HR(x')$ and \eqref{eq:limit} we have $HR(x',t)_\z$ for all $\z>0$ in this case. Thus (2) holds for $X \cup \{x'\}$ as well.

By inspection, (1) and (2) hold for the set $X = \{w \in W \;| \;\ell(w) \le 2\}$. Hence by induction we obtain (1) and (2) for $X=W$. We have already
explained why this implies all of the theorems in \S\ref{sec:results}.

\subsection{Note to the reader} In order to keep this paper short and
have it cite only available sources, we have written it in the
language of \cite{S3}. However \cite{S3} is not an easy
paper, and we make heavy use of its results. We did not discover the results of this paper in this language, but rather in the diagrammatic language of \cite{EW1} and \cite{EW2}.
These papers also provide alternative proofs of the requisite results from \cite{S3}.

\subsection{Acknowledgements} The second author would like to thank
Mark Andrea de Cataldo, Peter Fiebig, Luca Migliorini and Wolfgang Soergel for
useful discussions. Part of this work was 
completed when the first author visited the MPIM, and the second
author visited Columbia University. Both authors would like to
thank both institutions. We would also like to thank Mikhail Khovanov
for encouraging our collaboration. Finally, thanks to Henning Haahr
Andersen, Nicolas Libedinsky, Walter Mazorchuk, Patrick Polo and the
referee for detailed comments and suggestions.

The second author presented the results of this paper at the conference ``Lie algebras and applications'' in Uppsala in September, 2012.

The second author dedicates his work to the memory of
Leigh, who would not have given two hoots if certain
polynomials have positive coefficients!


\section{Lefschetz linear algebra}

\label{sec:lla}

  Let $H = \bigoplus_{i \in \ZM} H^i$ be a graded finite dimensional
real vector space equipped with a non-degenerate symmetric 
bilinear form \[ \langle -, - \rangle_H : H \ot_{\RM} H \to \RM \] which is \emph{graded} in the sense that $\langle H^i, H^j \rangle = 0$ unless $i = -j$.

Let $L \colon H^\bullet \to H^{\bullet + 2}$ denote an operator of degree 2. We may also write $L \in \Hom(H,H(2))$, where $(2)$ indicates a grading shift. We say that $L$ is a
\emph{Lefschetz operator} if $\langle Lh, h' \rangle = \langle h, Lh'
\rangle$ for all $h, h' \in H$. We assume from now on that $L$ is a
Lefschetz operator. We say that $L$ \emph{satisfies
  the hard Lefschetz theorem} if $L^i : H^{-i} \to 
H^{i}$ is an isomorphism for all $i \in \ZM_{\ge 0}$. For $i \ge 0$
set \[ P^{-i}_L := \ker L^{i+1} \subset H^{-i}. \] We call $P^{-i}_L$
the \emph{primitive subspace} of $H^{-i}$ (with 
respect to $L$). If $L$ satisfies the hard Lefschetz theorem then we
have a decomposition \[ H = \bigoplus_{i \ge 0 \atop 0 \le j
  \le i} L^j P_L^{-i}. \] This is the 
\emph{primitive decomposition} of $H$.

For each $i \ge 0$ we define the \emph{Lefschetz form} on $H^{-i}$
via \[ (h,h')^{-i}_L := \langle h, L^i h' \rangle. \] All Lefschetz
forms are non-degenerate if and only if $L$ satisfies the hard
Lefschetz theorem, because $\langle -, - \rangle$ is non-degenerate by
assumption. Because $L$ is a Lefschetz operator we have $(h,h')^{-i}_L =
(Lh,Lh')^{-i+2}_L$ for all $i \ge 2$ and $h, h' \in H^{-i}$. If $L$ satisfies the hard Lefschetz theorem then the primitive
decomposition is orthogonal with respect to the Lefschetz forms.

We say that $H$ is \emph{odd} (resp. \emph{even}) if
$H^\textrm{even} = 0$ (resp. $H^\textrm{odd} = 0$). Recall that a
bilinear form $(-,-)$ on a real vector space is said to be
\emph{$+1$ definite} (resp. \emph{$-1$ definite}) if $(v,v)$ is strictly positive (resp. negative) for all non-zero vectors $v$.

Let $H$ and $L$ be as above, and assume that $L$ satisfies the hard Lefschetz theorem. Assume that $H$ is either even or odd and set $j = 0$ if $H$ is even,
and $j = 1$ if $H$ is odd. We say that $H$ and $L$ \emph{satisfy the Hodge-Riemann bilinear relations} if there exists $\e \in \{\pm 1\}$ such that the restriction of
$(-,-)^{-i}_L$ to each primitive component $P_L^{-i} \subset H^{-i}$
is $\e(-1)^{(-i+j)/2}$ definite for all $i \le 0$. We fix the
ambiguity of global sign as follows: if $H$ and $L$ satisfy 
the Hodge-Riemann bilinear relations, we say that the
Hodge-Riemann bilinear relations are satisfied with the \emph{standard
  sign} if the Lefschetz form is positive definite on the lowest
non-zero degree of $H$ (which is necessarily primitive because
$L$ satisfies the hard Lefschetz theorem).

In order to avoid having to always specify if a vector space is even
or odd we will adopt the following convention: the statement that a
form on a space $P$ is $(-1)^{m/2}$-definite has the above meaning if
$m$ is even, and means that $P = 0$ if $m$ is odd.

If $H$ and $L$ satisfy the Hodge-Riemann bilinear relations then in particular each Lefschetz form $( -, -)^{-i}_L$ is non-degenerate. Moreover, its signature is easily determined from
the graded rank of $H$. (Use the fact that the primitive decomposition is orthogonal.) In fact, the Hodge-Riemann bilinear relations are equivalent to a statement about the signatures of
all Lefschetz forms.

In the sequel, we will need to consider families of Lefschetz operators (keeping $H$ and the form $\langle -, - \rangle$ fixed). It will be important to be able to
decide whether any or all members of the family are Hodge-Riemann. The following elementary lemma will provide an invaluable tool:

\begin{lem} \label{lem:family} Let $a < b$ in $\RM$ and let $\phi :
  [a,b] \to \Hom(H,H(2))$ be a continuous map (in the standard
  Euclidean topologies) such that $\phi(t)$ is a
  Lefschetz operator satisfying the hard Lefschetz theorem for all $t
  \in [a,b]$. If there exists $t_0 \in [a,b]$ such that $\phi(t_0)$ satisfies the Hodge-Riemann bilinear relations, then all $\phi(t)$ for $t \in [a,b]$ satisfy the Hodge-Riemann bilinear
relations. \end{lem}

\begin{proof} This follows from the fact that the signature of a continuous family of non-degenerate symmetric bilinear forms is constant. \end{proof}

In general it is difficult to decide whether the restriction of a non-degenerate bilinear form to a subspace stays non-degenerate. However, it is obvious that the restriction of a 
definite form is non-degenerate. This basic fact plays a crucial role in this paper. The following lemma extends this observation to certain $L$-stable subspaces of $H$.

\begin{lem} \label{lem: HR subspace} Assume that $H$ and $L$
  satisfy the Hodge-Riemann bilinear relations. Let $V \subset H$
  denote an $L$-stable graded subspace such
that $\dim V^i = \dim V^{-i}$. Then $V$ and $L$ satisfy the Hodge-Riemann bilinear relations (with respect to the restriction of $\langle-,-\rangle$ to $V$). \end{lem}

\begin{proof} (Sketch: the reader should provide a proof.) By symmetry of Betti numbers and hard Lefschetz, $V$ admits a primitive decomposition, and the result follows.
\end{proof}

The following lemma will serve as a substitute
for the weak Lefschetz theorem:

\begin{lem} \label{lem: weak lefschetz substitute} Suppose that we
  have a map of graded $\RM[L]$-modules ($\deg L = 2$) \[ \phi: V \to
  W(1) \] such that
 \begin{enumerate} \item $\phi$ is
injective in degrees $\le -1$, \item $V$ and $W$ are equipped with
graded bilinear forms $\langle -, - \rangle_V$ and $\langle -, - \rangle_W$ such that $\langle \phi(\a), \phi(\b) \rangle_W = \langle \a, L \b
\rangle_V$ for all $\a, \b \in V$, \item $W$ satisfies the Hodge-Riemann bilinear relations. \end{enumerate} Then $L^i : V^{-i} \to V^i$ is injective for $i\ge 0$. \end{lem}

\begin{proof} For $i=0$ the statement is vacuous. Choose $0 \ne \a \in V^{-i}$ with $i \ge 1$, and consider $0 \ne \phi(\a) \in W^{-i+1}$. If $0 \ne L^i \phi(\a) = \phi(L^i \a)$ then
$L^i \a \ne 0$. Alternatively, if $L^i \phi(\a) = 0$ then $\phi(\a)$
is primitive. Hence \[ (\phi(\a), \phi(\a))^{-i+1}_L = \langle \phi(\a), L^{i-1} \phi(\a) \rangle_W = \langle \a,
L^i \a \rangle_V \] is either strictly negative or positive by the Hodge-Riemann bilinear relations. In any case $L^i \a \ne 0$. Hence $L^i : V^{-i} \to V^i$ is injective as claimed.
\end{proof}

When $\dim(V^{-i}) = \dim(V^i)$ for all $i$, this lemma implies the hard Lefschetz theorem for $V$.

\begin{remark} \label{remark: standard sign} Suppose we are in the situation of the above lemma, that $-\ell$ is the lowest degree of $W$, and that $-(\ell+1)$ is the lowest degree
of $V$. The above proof indicates that $(-,-)^{-(\ell+1)}_L$ is $\pm$
definite on $V^{-(\ell+1)}$, with the same sign as on $W^{-\ell}$. In
particular, if $V$ also satisfies the Hodge-Riemann
bilinear relations, then $V$ has the standard sign if and only if $W$
has the standard sign. \end{remark}

Finally, we will need the following lemma in
\S\ref{sec:RHR}. Let $H$, $\langle -, - \rangle$ and $L$ be as in the
first two paragraphs of this section, except that we no longer assume
that $\langle -, - \rangle$ is non-degenerate. Suppose that there
exists $d \in \ZM$ such that $L^i : H^{-d-i} \to H^{-d+i}$ is an
isomorphism, for all $i \ge 0$ (so $L$ satisfies the hard Lefschetz theorem
if and only if $d = 0$).

\begin{lem}\label{lem:shiftedvanishing}
If $d > 0$ then the Lefschetz form $(h, h')_L^{-i} :=
  \langle h, L^i h \rangle$ on $H^{-i}$ for $i \ge 0$
is zero. \end{lem}

\begin{proof} For $i \ge 0$ consider the ``shifted primitive spaces'' 
\[
Q_L^{-d-i} := \ker L^{i+1} \subset H^{-d-i}
\]
and set $Q_L^{-d-i} := 0$ if $i < 0$. Then our assumptions on $L$
guarantee that we have a ``shifted primitive decomposition''
\[
H^m = \bigoplus_{j \ge 0} L^j Q_L^{m-2j}.
\]

Fix a degree $m \le 0$ and fix $x \in Q_L^{m-2j}$ and $y \in Q_L^{m-2k}$ for some $j \ge k \ge 0$, so that $L^j x$ and $L^k y$ are in degree $m$. Then
\[
(L^jx, L^ky)_L^{m} = \langle x, L^{j+k-m}y \rangle = 0.
\]
This follows because $y \in \ker L^{2k-d-m+1}$ and $2k-d-m+1 \le j+k-m$, thanks to the assumption $d > 0$.
\end{proof}


\section{The Hecke algebra and Soergel bimodules} \label{sec:background}

\subsection{Coxeter systems} \label{sec:cox}

Fix a Coxeter system $(W,S)$ and for simple reflections $s, t \in S$
denote by $m_{st} \in \{ 2, 3, \dots, \infty \}$ the order of $st$. We
denote the length function on $W$ by $\ell$ and the Bruhat order by
$\le$.

An \emph{expression} is a word $\un{z} =
s_1s_2\cdots s_m$ in $S$. An expression will always be denoted by an underlined roman letter. Omitting the underline will denote the product in the Coxeter group. An expression $\un{z} =
s_1s_2 \cdots s_m$ is \emph{reduced} if $m = \ell(z)$.

Let us fix a finite dimensional real vector space $\hg$ together with linearly independent subsets $\{ \a_s \}_{s \in S} \subset \hg^*$ and $\{ \a_s^\vee \}_{s \in S} \subset \hg$ such
that
\[ \a_s(\alpha_t^\vee) = -2 \cos(\pi/m_{st}) \quad \text{for all $s, t
  \in S$.} \]
In addition, we assume that $\hg$ is of minimal
  dimension with these properties.
The group $W$ acts on $\hg$ by $s \cdot v = v -
\alpha_s(v)\alpha_s^\vee$. This action is \emph{reflection faithful}
in the sense of \cite[Definition 1.5]{S3} (see \cite[Proposition 2.1]{S3}).

\begin{remark} We have assumed that the representation $\hg$ is
  reflection faithful so that the theory of \cite{S3} is available. It
  was shown by Libedinsky \cite{Li} that Soergel's conjecture for
  $\hg$ is equivalent to Soergel's conjecture for the geometric
  representation. We discuss the choice of representation in detail in
  \cite{EW2} where we give alternative proofs of the results of
  \cite{S3} which are valid when $\hg$ is any ``realization'' of $W$.
\end{remark}

Let $R$ be the coordinate ring of $\hg$, graded so that its
linear terms $\hg^*$ have degree $2$. We denote by $R^+$ the ideal
of elements of positive degree.
Clearly $W$ acts on $R$. For $s \in S$ we write $R^s$ for the
subring of invariants under $s$.

Because the vectors $\{ \alpha_s^\vee \}_{s \in S}$ are linearly
independent the intersection of the open half spaces
\[ \bigcap_{s \in S}
\{ v \in \hg^* \; | \; v(\alpha_s^\vee) >
0\} \subset \hg^*\]is non-empty.
We fix once and for all an
element $\rho \in \hg^*$ in this intersection. That is, we fix $\rho$
such that $\rho(\alpha_s^\vee) >
0$ for all $ s\in S$.
The following positivity property of the representation $\hg$
plays an important role below (see \cite[V.4.3]{Bo} or \cite[Lemma
5.13]{Hu} and the proof of \cite[Proposition 2.1]{S3}):
\begin{equation} \label{eq:pos}
(w\rho)(\alpha_s^{\vee}) > 0 \Leftrightarrow sw > w.
\end{equation}


\subsection{The Hecke algebra} \label{sec:hecke}
References for this section are \cite{KL1} and \cite{SoKL}.
Recall that the Hecke algebra $\HC$ is the algebra with free $\ZM[v^{\pm 1}]$-basis given by symbols $\{ H_x \; | \; x \in W \}$ with multiplication determined by \[ H_xH_s :=
\begin{cases} H_{xs} & \text{if $xs > x$,} \\ (v^{-1} - v)H_x + H_{xs} & \text{if $xs < x$.} \end{cases} \] Given $p \in \ZM[v^{\pm 1}]$ we write $\ov{p(v)} := p(v^{-1})$. We can
extend this to an involution of $\HC$ by setting $\ov{H_x} = H^{-1}_{x^{-1}}$. Denote the Kazhdan-Lusztig basis of $\HC$ by $\{ \un{H}_x \; | \; x \in W \}$. It is characterised by
the two conditions
 \begin{enumerate} \item[i)] $\ov{\un{H}_x} = \un{H}_x$ \item[ii)] $\un{H}_x \in H_x + \sum_{y < x} v\ZM [v]H_y$ \end{enumerate}
for all $x \in W$. For example, if $s
\in S$ then $\un{H}_s = H_s + vH_{\id}$.

\begin{remark} \label{rem:oldKL}
  In the notation of \cite{KL1} we have $v = q^{-1/2}$, $H_x =
  v^{\ell(x)}T_x$ and $\un{H}_x = C_x^\prime$.  If we write $\un{H}_x = \sum h_{y,x}H_y$ then
  $v^{\ell(x)-\ell(y)}P_{y,x}(v^{-2}) = h_{y,x}$.
\end{remark}

Consider the $\ZM[v,v^{-1}]$-linear trace $\e : \HC \to \ZM[v^{\pm
  1}]$ given by $\e(H_w)= \delta_{\id,w}$. Define a bilinear form 
\begin{align*}
(-,-): \HC \times \HC & \to \ZM[v^{\pm 1}] \\
(h,h') & \mapsto\e(a(h)h')
\end{align*}
where $a$ is the anti-involution of $\HC$ determined by $a(v) = v$ and
$a(H_x) = H_{x^{-1}}$. One checks easily that
\begin{enumerate}
\item[i)] $(ph, qh') = {p}q(h,h')$ for all $p, q \in
  \ZM[v^{\pm 1}]$ and $h,h' \in \HC$,
\item[ii)] $(h\un{H}_s, h') = (h, h'\un{H}_s)$, $(\un{H}_sh, h') = (h,
  \un{H}_sh')$ for all $h, h' \in \HC$ and $s \in S$.
\end{enumerate}
A straightforward induction shows $(H_x,
H_y) =  \delta_{xy}$, which we could have used as the
definition of $(-,-)$.

An important property of this pairing (used repeatedly below) is that $(\un{H}_x,\un{H}_y) \in v \ZM[v]$ when $x \ne y$, and $(\un{H}_x,\un{H}_x) \in 1 + v \ZM[v]$.

\begin{remark} This is not the form used in \cite{EW2}, which is more natural when one only considers Soergel bimodules. In this paper we also consider $\Delta$- and $\nabla$-filtered
bimodules, for which the above form is more convenient. \end{remark}

\subsection{Bimodules}

We work in the abelian category of finitely generated graded $R$-bimodules. All morphisms
preserve the grading (i.e. are homogeneous of degree 0). Given a graded $R$-bimodule $B = \bigoplus_{i \in \ZM} B^i$, we denote by $B(1)$ the shifted bimodule: $B(1)^i =
B^{i+1}$. We write $\Hom(-,-)$ for degree zero morphisms between
bimodules (the morphisms in our category). For any two bimodules $M$ and $N$ set \[ \Hom^{\bullet}(M,N) = \bigoplus_{i \in \ZM}\Hom(M, N(i)). \] Given a
polynomial $p = \sum_{i \in \ZM} a_i v^i \in \ZM_{\ge 0}[v^{\pm 1}]$,
we let $B^{\oplus p}$ denote the bimodule $\oplus_{i \in \ZM}
B(i)^{\oplus a_i}$. Given bimodules $B$ and $B'$ we write $B' \summand
B$ to mean that $B'$ is a direct summand of $B$.  Throughout ``embedding'' means
``embedding as a direct summand''.

The category of $R$-bimodules is a monoidal category under tensor
product. Given $R$-bimodules $B$ and $B'$ we denote their tensor product by
juxtaposition: $BB' := B \otimes_R B$.

Throughout this paper, we have arbitrarily chosen the right action to be special for many constructions. For instance, for a bimodule $B$ we will often consider $\ov{B} = B \ot_R \RM$;
here $\RM = R/R^+$ is the $R$-module where all positive degree polynomials vanish.

We define the dual of an $R$-bimodule $B$ by $ \DM B :=
\Hom^\bullet_{-R}(B, R)$. Here, $\Hom^\bullet_{-R}(-,-)$ denotes
homomorphisms of all degrees between right $R$-modules. We make $\DM B$ 
into an $R$-bimodule via $(r_1fr_2)(b) = f(r_1br_2)$ (where $f \in
\DM B$, $r_1, r_2 \in R$ and $b \in B$). Suppose that $B$ is
finitely generated and graded free as a right $R$-module, so that $B
\cong R^{\oplus p}$ as a
right $R$-module (for some $p \in \ZM_{\ge 0}[v^{\pm 1}]$). Then $\DM B \cong R^{\oplus
  \overline{p}}$. In particular, if $B \cong \DM B$ then $\overline{B}
\cong \overline{B}^*$ as graded $\RM$-vector spaces and $\dim
(\ov{B})^{-i} = \dim (\ov{B})^{+i}$.

We say that an $R$-valued form $\langle -, - \rangle_B$ on a graded
$R$-bimodule $B$ is \emph{graded} if $\deg \langle b, b' \rangle_B =
\deg b + \deg b'$ for homogeneous $b,b'$. A form $\langle -, -
\rangle_B$ is \emph{invariant} if it is graded and $\langle r b, b' \rangle = \langle b, rb' \rangle$ and $\langle br, b' \rangle = \langle b, b'r \rangle = \langle b, b'
\rangle r$ for all $b, b' \in B$ and $r \in R$. (Note the left/right asymmetry).
The space of invariant forms is
isomorphic to the space of $R$-bimodule maps $B \to \DM B$.
We say that an invariant form $\langle -, - \rangle_B$ on a bimodule $B$
is \emph{non-degenerate} if it induces an isomorphism $B \to \DM
B$. This is stronger than assuming non-degeneracy in the usual sense ($\langle b, b'
\rangle = 0$ for all $b' \in B$ implies $b = 0$). An invariant
form $\langle -, - \rangle_B$ on $B$ induces a form $\langle - , -
\rangle_{\overline{B}}$ on $\overline{B}$ by defining $\langle f, g
\rangle_{\overline{B}}$ to be the image of $\langle f, g \rangle_B$ in
$\RM = R/R^+$.

Suppose that $B$ is free of finite rank as a right $R$-module (as will be the case
for all bimodules considered below). Then  an invariant form
$\langle -, - \rangle_B$ is non-degenerate if and only if $\langle -,
- \rangle_{\overline{B}}$ gives a graded (in the sense of \S\ref{sec:lla}) non-degenerate form on the graded
vector space $\overline{B}$, as follows from the graded Nakayama lemma.

\subsection{Bott-Samelson bimodules} \label{sec:BSbimodules}

For any simple reflection $s \in S$ set $B_s := R \ot_{R^s}
R(1)$. It is an $R$-bimodule with respect to left and right
multiplication by $R$.
Consider the elements \[ c_{\id} := 1 \ot 1 \in B_s, \quad
\quad c_s := \frac{1}{2}(\a_s \ot 1 + 1 \ot \a_s) \in B_s
\] 
of degrees $-1$ and $1$ respectively. Then $\{ c_\id, c_s \}$ gives a basis for $B_s$ as a right (or left) $R$-module, and one has the relations \begin{gather} \label{eq:invariance} r \cdot c_s = c_s \cdot r \\
\label{eq:demazure} r \cdot c_{\id} = c_{\id} \cdot sr + \partial_s(r) \cdot c_s \end{gather} for all $r \in R$. Here, $\partial_s$ is the \emph{Demazure operator}, given by \[
\partial_s(r) = \frac{r - sr}{\a_s} \in R. \]
For any expression $\un{x} = st\cdots u$ we denote by $BS(\un{x})$ the corresponding Bott-Samelson bimodule: \[ BS(\un{x}) := B_sB_t\cdots
B_u = B_s \ot_R B_t \ot_R \cdots \ot_R B_u. \] Given elements $b_s \in B_s$, $b_t \in B_t$, \dots, $b_u \in B_u$ we denote the corresponding tensor simply by juxtaposition $b_sb_t \cdots
b_u := b_s \ot b_t \ot \cdots \ot b_u.$ For any subexpression $\un{\e}$ of $\un{x}$ (that is, $\un{\e} = \e_s \e_t \cdots \e_u$ with $\e_v \in \{ \id ,v \}$ for all $v \in S$) we can
consider the element \[ c_{\un{\e}} := c_{\e_s} c_{\e_t} \cdots c_{\e_u} \in BS(\un{x}). \] One may check that the set $\{ c_{\un{\e}} \}$ gives a basis for $BS(\un{x})$ as a right (or
left) $R$-module, as $\un{\e}$ runs over all subexpressions of $\un{x}$.

In the following the element $c_\top := c_s c_t \cdots c_u \in BS(\un{x})$ will play an important role. Given $b \in BS(\un{x})$ we define $\Tr(b) \in R$ to be the coefficient of $c_\top$
when $b$ is expressed in the basis $\{ c_{\un{\e}} \}$ of $BS(\un{x})$ as a right $R$-module.

Clearly, $BS(\un{x}) \cong (R \ot_{R^s} R \ot_{R^t} \cdots
\ot_{R^u} R) (d)$ where $d$ is the length of the expression $\un{x}$. It
follows that 
$BS(\un{x})(-d)$ is a commutative ring, with term-wise multiplication. For example, $(f \ot g) \cdot (f' \ot g') = ff' \ot gg'$ gives the multiplication on $B_s(-1) = R \ot_{R^s} R$. Let
us observe the following multiplication rules in $B_s$: \begin{gather} c_{\id} \cdot c_{\id} = c_{\id} \\ c_{\id} \cdot c_s = c_s \\ c_s \cdot c_s = c_s \a_s. \end{gather} We
define an invariant symmetric form $\langle -, - \rangle_{BS(\un{x})}$ on $BS(\un{x})$ via \[ \langle b, b' \rangle := \Tr(b \cdot b'). \] We call $\langle -, -
\rangle_{BS(\un{x})}$ the \emph{intersection form} on
$BS(\un{x})$. It induces a symmetric $\RM$-valued form $\langle -, -
\rangle_{\overline{BS(\un{x})}}$ on $\ov{BS(\un{x})}$. If we write
$\Tr_\RM$ for the composition of $\Tr$ with the quotient map $R \to
\RM = R/R^+$ then we have
\[
\langle b, b' \rangle_{\ov{BS(\un{x})}} = \Tr_\RM(b \cdot b').
\]
for all $b, b' \in \ov{BS(\un{x})}$. Because $R \ot_{R^s} R \ot_{R^t}
\cdots \ot_{R^u} R$ is a commutative 
ring, multiplication by any degree $2$ element $z$ of this ring 
gives a Lefschetz operator on $BS(\un{x})$. In other words, $\langle
z\alpha,\beta \rangle = \langle \a,z\b \rangle$ for any $\a, \b \in BS(\un{x})$.

For $s \in S$ let $\mu : B_s \to R(1) : f\ot g \mapsto fg$ denote
the multiplication map. For $1 \le i \le m$ set $x_{\widehat{i}} = s_1 \cdots \widehat{s_i}
\cdots s_m$ ($\widehat{s_i}$ denotes omission). Consider the canonical maps 
\begin{align*}
\Br_i : & BS(\un{x}) \to BS(\un{x})(2) : b_1 \cdots b_i \cdots b_m
\mapsto b_1 \cdots (b_ic_{s_i}) \cdots b_m \\
\phi_i : & BS(\un{x}) \to BS(\un{x}_{\widehat{i}})(1) : b_1 \cdots b_i \cdots
b_m \mapsto b_1 \cdots \mu(b_i) \cdots b_m \\
\chi_i : & BS(\un{x}_{\widehat{i}}) \to BS(\un{x})(1) : b_1 \cdots
b_{i-1}b_{i+1} \cdots b_m \mapsto b_1 \cdots b_{i-1}c_{s_i}b_{i+1} \cdots b_m.
\end{align*}
By \eqref{eq:invariance} we have $\Br_i = \chi_i \circ \phi_i$.

\begin{lem} \label{lem: breaking lefschetz}
As endomorphisms of $BS(\un{x})$ we have:
\[
\rho \cdot (-) = \sum_{i = 1}^m (s_{i-1} \cdots s_1 \rho
)(\a_{s_i}^\vee)  \chi_i \circ \phi_i + (-) \cdot x^{-1}\rho
\]
Here $\rho \cdot (-)$ denotes left multiplication by $\rho$ and $(-) \cdot x^{-1}\rho$ denotes right
multiplication by $x^{-1}\rho$.
\end{lem}

\begin{proof} This is an immediate consequence of \eqref{eq:demazure}.
\end{proof}

\subsection{Soergel bimodules} \label{sec:SB}

By definition, a \emph{Soergel bimodule} is an object in the additive
Karoubian subcategory $\BC$ of graded $R$-bimodules generated by Bott-Samelson bimodules and their shifts. In other words,
indecomposable Soergel bimodules are the indecomposable $R$-bimodule
summands of Bott-Samelson bimodules (up to shift).

It is a theorem of Soergel \cite{S3} that, given any reduced
expression $\un{x}$ for $x \in W$ there is a unique (up to
isomorphism) indecomposable summand $B_x$ of $BS(\un{x})$ which does
not occur as a direct summand of $BS(\un{y})$ for any expression
$\un{y}$ of length less than $\ell(x)$. Moreover, $B_x$ does not
depend (up to isomorphism) on the choice of reduced expression
$\un{x}$. The bimodules $B_x$ for $x \in W$ give representatives for
the isomorphism classes of indecomposable Soergel bimodules, up to
shift.

Denote by $[\BC]$ the split Grothendieck group of $\BC$. That is,
$[\BC]$ is the abelian group generated by symbols $[B]$ for all
objects $B \in \BC$ subject to the relations $[B] = [B'] + [B'']$
whenever $B \cong B' \oplus B''$ in $\BC$. We make $[\BC]$ into a
$\ZM[v^{\pm 1}]$-module via $p [M] := [M^{\oplus p}]$ for $p
\in \ZM_{\ge 0}[v^{\pm 1}]$ and $M \in \BC$. Because $\BC$ is monoidal, $[\BC]$ is a
$\ZM[v^{\pm 1}]$-algebra. The above results imply that $[\BC]$ is
free as a $\ZM[v^{\pm 1}]$-module, with basis $\{ [B_x] \; | \; x \in W
\}$. In fact one has \cite[Theorem 1.10]{S3}:

\begin{thm}[Soergel's categorification theorem] \label{thm: Soergel
    Catfn Thm} There is an isomorphism of $\ZM[v^{\pm 1}]$-algebras
\[ \HC \simto  [\BC]
\]
fixed by $\un{H}_s  \mapsto [B_s]$.\end{thm}

We now describe Soergel's construction of an inverse to the
isomorphism $\HC \simto [\BC]$. To do this it is natural
to consider certain filtrations ``by support'' (see \cite[\S3 and \S5]{S3}). For $x\in W$ consider
the linear subspace (or ``twisted graph'')
\[
\Gr(x)=\{(xv,v)\,\vert\, v\in \hg\}\subseteq
\hg\times \hg
\]
which we view as a subvariety in $\hg \times \hg$. For any subset $A$
of $W$ consider the corresponding union
\[\Gr(A)=\bigcup_{x\in A}\Gr(x)\subseteq \hg\times \hg.\]
Let us identify $R \otimes_\RM R$ with the regular functions on
$\hg \times \hg$. Any $R$-bimodule can be viewed as an $R
\otimes_{\RM} R$-module (because $R$ is commutative) and hence as a
quasi-coherent sheaf on $\hg \times \hg$. For example, one may check
that the bimodule $R_x$ corresponding to the structure sheaf on
$\Gr(x)$ has the following simple description: $R_x
\cong R$ as a left module, and the right action is twisted by $x$: $m
\cdot r = m(xr)$ for $m \in R_x$ and $r \in R$.

Given any subset $A \subseteq W$ and $R$-bimodule $M$ we define
\[
\Gamma_A M := \{ m \in M \; | \; \supp m \subseteq \Gr(A) \}
\]
to be the subbimodule consisting of elements whose support is contained in $\Gr(A)$. 
Given $x \in W$ we will abuse notation and write $\le x$ for the set
$\{ y \in W \; | \; y \le x \}$ and similarly for $<x$, $\ge x$ and $>
x$. With this notation, we obtain functors $\Gamma_{\le x}$, $\Gamma_{< x}$,
$\Gamma_{\ge x}$ and $\Gamma_{> x}$. For example $\Gamma_{\le x} =
\Gamma_{\{ y \in W \; | \; y \le x \}}$.

For any $x \in W$ define  $\Delta_x := R_x(-\ell(x))$ and  $\nabla_x
:= R_x(\ell(x))$.
Given a finitely generated $R$-bimodule $M$ we say that $M$ \emph{has
  a $\Delta$-filtration} (resp. \emph{has a $\nabla$-filtration}) if
$M$ is supported on $\Gr_A$ for some finite subset $A \subset W$ and,
for all $x \in W$ we have isomorphisms
\[
\Gamma_{\ge x} M / \Gamma_{> x} M \cong \Delta_x^{\oplus
  h^{\Delta}_x(M)} 
\qquad ( \text{resp.} \quad \Gamma_{\le x} M / \Gamma_{< x} M \cong
\nabla_x^{\oplus  h^{\nabla}_x(M)} \; )
\]
for some polynomials $h^{\Delta}_x(M)$ (resp. $h^{\nabla}_x(M)$) in
$\ZM_{\ge 0}[v^{\pm 1}]$. If $M$ has a $\Delta$-filtration
(resp. $\nabla$-filtration) we define its \emph{$\Delta$-character}
(resp. \emph{$\nabla$-character}) in the Hecke algebra via
\[
\ch_\Delta M := \sum_{ x \in W}   h^{\Delta}_x(M) H_x
\qquad ( \text{resp.} \quad 
\ch_\nabla M := \sum_{ x \in W}   \overline{h^{\nabla}_x(M)} H_x).
\]
Note that $\ch_\Delta M(1) = v\ch_\Delta M$ whilst
$\ch_{\nabla} M(1) = v^{-1}\ch_{\nabla} M$.

By \cite[Propositions 5.7 and 5.9]{S3}, Soergel bimodules have both $\Delta$-
and $\nabla$-filtrations and by \cite[Bemerkung 6.16]{S3} we have
\[
\ch_\Delta B = \ov{\ch_\nabla B}
\]
for any Soergel bimodule. For any Soergel bimodule $B$ we set
\[
\ch(B) := \ch_\Delta(B).\]
By \cite[Theorem 5.3]{S3}, $\ch : [\BC] \to \HC$
gives an inverse of the isomorphism
$\HC \simto [\BC]$ of Soergel's categorification theorem.

Finally, Soergel has given a beautiful formula for the graded rank of
homomorphism spaces between Soergel bimodules in terms of $\Delta$ and
$\nabla$-characters. Given a finite dimensional
graded $\RM$-vector space $V = \bigoplus 
V^i$ we define \[ \gdim V = \sum (\dim V^i) v^{-i} \in \ZM_{\ge 0}[v^{\pm
  1}]. \] Our notation is chosen so that $\gdim (V^{\oplus p}) = p
\gdim V$ for $p \in \ZM_{\ge 0}[v^{\pm 1}]$. Given a free finitely generated
graded right $R$-module $M$ we set \[ \grk M := \gdim (M \ot_R
\RM). \]

\begin{thm}[Soergel's hom formula] \label{thm:hom} Suppose that $B$
  has a $\Delta$-filtration and $B' \in \BC$ or that $B \in \BC$ and
  $B'$ has a $\nabla$-filtration. Then $\Hom^\bullet(B, B')$ is a graded free
  right $R$-module of rank 
\[
\grk \Hom^\bullet(B, B') = \ov{(\ch_\Delta B,\ch_\nabla B')}.
\]
\end{thm}

If Soergel's conjecture holds for $B_x$ and $B_y$ then 
$\ch B_x = \un{H}_x$ and $\ch B_y = \un{H}_y$. Soergel's hom formula
then implies that $\Hom^\bullet(B_x,
B_y)$ is concentrated in degrees $\ge 0$, and $\dim \Hom(B_x,B_y) = \delta_{xy}$.

\subsection{Invariant forms on Soergel bimodules} \label{sec:invariant
  forms} Let $B$ denote a self-dual Soergel bimodule. Equipping $B$ with an invariant non-degenerate bilinear form $\langle -, - \rangle_B$ is the same as
giving an isomorphism $B \simto \DM B$. It is known (see \cite[Satz
6.14]{S3}) that each indecomposable Soergel bimodule is 
self-dual and hence admits a non-degenerate invariant form. Moreover, if Soergel's conjecture holds for $B_x$ then $\End(B_x) = \RM$ (as follows immediately from Soergel's hom formula).
This implies the following, which plays an important role in this paper:

\begin{lem} \label{lem:unique form} Suppose that Soergel's conjecture
  holds for $B_x$. Then $B_x$ admits an invariant form which is unique
  up to a scalar. Moreover, any non-zero
invariant form is non-degenerate. \end{lem}

\begin{proof} Giving an invariant form on $B_x$ is the same thing as
  giving a graded $R$-bimodule morphism $B_x \to \DM B_x$. By the remarks
  preceding the lemma, the space of such maps is one-dimensional and
  contains an isomorphism. Hence $B_x$ admits an invariant form
  $\langle - , - \rangle_{B_x}$, and all others are scalar multiples
  of $\langle -, - \rangle_{B_x}$. The lemma now follows.
\end{proof}

We now explain how Soergel bimodules may be inductively equipped with
invariant forms. Fix a Soergel bimodule $B$ and consider the two maps
$\a, \b : B \to BB_s = B \otimes_R B_s$ given by
\[
\a(b) := bc_{\id} \quad \text{and} \quad \b(b) := bc_s
\]
Note that $\b$ is a morphism of bimodules, whilst $\a$ is only a
morphism of left modules: by \eqref{eq:demazure} one has
\begin{equation} \label{eq:alphar}
\a(br) = \a(b)(sr) + \b(b) \partial_s(r)
\end{equation}
for $b \in B$ and $r \in R$.

Suppose that $B$ is equipped with an invariant form $\langle -, -
\rangle_B$. Then there is a unique invariant form $\langle -, -
\rangle_{BB_s}$ on $BB_s$, which we call the
\emph{induced form},
satisfying
\begin{align} \label{alphaalpha}
\langle \a(b), \a(b') \rangle_{BB_s} &= \partial_s(\langle b, b' \rangle_B)\\
\langle \a(b), \b(b') \rangle_{BB_s} = \langle b, b' \rangle_B &\text{ and }  \label{alphabeta}
\langle \b(b), \a(b') \rangle_{BB_s} = \langle b, b' \rangle_B \\
\langle \b(b), \b(b') \rangle_{BB_s} &= \langle b, b' \rangle_B \alpha_s \label{betabeta}
\end{align}
for all $b, b' \in B$. Indeed, if $e_1, \dots, e_m$ denotes a basis for $B$ as a right $R$-module then $\a(e_1), \dots, \a(e_m), \b(e_1), \dots, \b(e_m)$ is a basis for $BB_s$ and
the above formulas fix $\langle -, - \rangle_{BB_s}$ on this basis. It is straightforward to check that $\langle -, - \rangle_{BB_s}$ satisfies \eqref{alphaalpha}, \eqref{alphabeta}
and \eqref{betabeta} for all $b, b' \in B$, and that $\langle rb, b'
\rangle_{BB_s} = \langle b, rb' \rangle$ for all $b, b' \in B$ and $r
\in R$. Clearly $\langle -, - \rangle_{BB_s}$ is symmetric if $\langle
-, - \rangle$ is.

Suppose that $B$ is a summand of a Bott-Samelson bimodule $BS(\un{x})$. Then $B$ is equipped with an invariant symmetric form $\langle -, - \rangle_B$, obtained by restriction from the
intersection form on $BS(\un{x})$. There are now two ways to equip $BB_s$ with an invariant form: either via the induced form as above, or by viewing $BB_s$ as a summand of $BS(\un{x})B_s
= BS(\un{x}s)$ and considering the restriction of the intersection
form. It is an easy exercise to see that these two forms agree, which motivates
the above formulas. If we apply this for $B =
BS(\un{x})$ we conclude that the intersection form on
$BS(\un{x})$ can also be obtained by starting with the canonical
multiplication form on $R$, and iterating the construction of the
induced form.

\begin{lem} \label{lem:induced non-degen} Suppose that $B$ is an
  $R$-bimodule which is equipped with an invariant form $\langle -, -
  \rangle_B$. Assume that $B$ is free as a right $R$-module and that $\langle -, -
  \rangle_B$ is non-degenerate. Then $\langle -, - \rangle_{BB_s}$ is
  non-degenerate. \end{lem}

\begin{proof} Because $\langle -, - \rangle_B$ is non-degenerate 
 and $B$ is free as a right $R$-module we
  can fix a basis $e_1, \dots, e_m$ and dual basis $e_1^*, \dots,
  e_m^*$ for $B$ as a right $R$-module. Then
\[\a(e_1), \dots, \a(e_m),
  \b(e_1), \dots, \b(e_m)\]and
\[
  \b(e_1^*), \dots, \b(e_m^*),
\a(e_1^*), \dots, \a(e_m^*),
\]are bases for $BB_s$ as a right
  $R$-module. Now \eqref{alphaalpha}, \eqref{alphabeta} and
  \eqref{betabeta} show that the matrix of $\langle -, -
  \rangle_{BB_s}$ in this pair of bases has the form
\begin{gather*}
\left ( \begin{matrix} I_m & \a_s I_m \\
0 & I_m \end{matrix} \right )
\end{gather*}
where $I_m$ denotes the $m \times m$ identity matrix. The zero matrix in the lower left arises because $\partial_s(1)=0$.
Hence $\langle -, - \rangle_{BB_s}$ is non-degenerate as
  claimed.
\end{proof}

\begin{cor} The intersection form on a Bott-Samelson bimodule is non-degenerate. \end{cor}

The following positivity calculation is not entirely necessary for the proofs
below. However it does give a simple explanation of why the global sign in the
Hodge-Riemann bilinear relations is correct.

\begin{lem} \label{lem:BSpositivity} The Lefschetz form $(-,-)^{-\ell(x)}_\rho$ on $\ov{BS(\un{x})}^{-\ell(x)} \cong \RM$ is positive-definite, when $\un{x}$ is a reduced expression. \end{lem}

\begin{proof} Let $c_{\bot} := c_{\id} c_{\id} \cdots c_{\id}$, which spans $BS(\un{x})^{-\ell(x)}$. We claim that $\rho^{\ell(x)} c_{\bot} = N c_{\top} \in \ov{BS(\un{x})}$ for some
$N>0$, which will imply the result. We induct on $\ell(x)$. The result is clear when $\ell(x)=0$.

By Lemma \ref{lem: breaking lefschetz} we have
\[ \rho \cdot c_{\bot} =
\sum_i (s_{i-1} \cdots s_1 \rho )(\a_{s_i}^\vee) \chi_i(c_{\bot}) +
c_{\bot} \cdot (x^{-1} \rho) \] inside $BS(\un{x})$.
Note that $(s_{i-1} \cdots s_1 \rho )(\a_{s_i}^\vee)$ is positive for all $i$,
by our positivity assumption on $\rho$ and the fact that $\un{x}$ is a
reduced expression. The final term clearly vanishes in $\ov{BS(\un{x})}$, so it remains to see what happens when $\rho^{\ell(x)-1}$ is applied to every other term.

Suppose that $\un{x}_{\widehat{i}}$ is a reduced expression. Then by induction $\rho^{\ell(x)-1} c_{\bot} = N_i c_{\top} \in \ov{BS(\un{x}_{\widehat{i}})}$ for some $N_i>0$. Clearly
$\chi_i(c_{\top})=c_{\top}$, so $\rho^{\ell(x)-1} \chi_i(c_{\bot}) = N_i c_{\top} \in \ov{BS(\un{x})}$.

Suppose that $\un{x}_{\widehat{i}}$ is not a reduced expression. In
this case,
\[
BS(\un{x}_{\widehat{i}}) \cong \bigoplus B_z^{\oplus p_z}
\]
with all $z$ appearing on the right hand side satisfying
$\ell(z)<\ell(x)-1$ and $p_z \in \ZM_{\ge 0}[v^{\pm 1}]$. For 
degree reasons $\rho^{\ell(x)-1}$ vanishes on $\ov{B_z}$ for any such
$z$, and therefore vanishes identically on
$\ov{BS(\un{x}_{\widehat{i}})}$. Therefore, $\rho^{\ell(x)-1} 
\chi_i(c_{\bot}) = 0$ for such $i$.

Therefore, $\rho^{\ell(x)} c_{\bot} = \left (\sum_i (s_{i-1} \cdots s_1 \rho
)(\a_{s_i}^\vee) N_i \right )  c_{\top} \in \ov{BS(\un{x})}$, with 
\[
\sum_i (s_{i-1} \cdots s_1 \rho )(\a_{s_i}^\vee) N_i > 0. \qedhere
\] \end{proof}

The following simple observation was promised in the introduction:

\begin{lem} \label{lem:HRembedding} If $S(x)$ holds, then $HR(x)$ and $HR(\un{x})$ are equivalent, for any reduced expression $\un{x}$. \end{lem}

\begin{proof} 	Obviously $B_x^i = 0$ for $i < -\ell(x)$. By considering
  the $\nabla$-character of $B_x$ it is easy to see that 
  $B_x^{-\ell(x)}$ is one dimensional. Hence any embedding $B_x \hookrightarrow
  BS(\un{x})$ induces an isomorphism $B_x^{-\ell(x)} \simto
  BS(\un{x})^{-\ell(x)} = \RM( c_{\bot})$.
	
	Given $S(x)$, Lemma \ref{lem:unique form} implies that the restriction of the intersection form on $BS(\un{x})$ to $B_x$ must be a scalar multiple of the intersection form on $B_x$. The Lefschetz form on $\ov{BS(\un{x})}^{-\ell(x)}$ is positive definite, and hence this scalar must be positive. Now $HR(x)$ and $HR(\un{x})$ are equivalent. \end{proof}

\section{The embedding theorem} \label{sec:embedding}

In this section we fix $x \in W$ and $s \in S$ with $xs > x$,
and we assume $S(y)$ and $HR(\un{y})$ for all $y < xs$. By
$HR(\un{y})$, if we choose an embedding $B_y \summand BS(\un{y})$ then
the restriction of the intersection form on $BS(\un{y})$ to $B_y$
yields a non-degenerate invariant form $\langle -,- \rangle_{B_y}$ on
$B_y$ which satisfies the Hodge-Riemann bilinear relations. 
Let us also fix a generator $c_\bot$ of the one-dimensional vector space
$B_y^{-\ell(y)}$. Then $HR(\un{y})$ implies 
\begin{equation} \label{eq:ynorm} \langle \rho^{\ell(y)} \cdot c_\bot, c_\bot \rangle_{B_y} = N
\end{equation} for some $0 < N \in \RM$.

Similarly, we fix an embedding $B_x \summand BS(\un{x})$ which induces
a non-degenerate form $\langle -, - \rangle_{B_x}$ on $B_x$. As discussed
in \S\ref{sec:invariant forms}, this induces a non-degenerate
invariant symmetric form $\langle -,- \rangle_{B_xB_s}$ on 
$B_x B_s$, compatible with the induced embedding $B_x B_s \summand
BS(\un{x})B_s = BS(\un{x}s)$. 

Having fixed these forms on $B_y$ and $B_xB_s$ we obtain a canonical identification \[ \Hom(B_y, B_xB_s) \simto \Hom(B_xB_s, B_y) \] sending $f \in \Hom(B_y, B_xB_s)$ to its adjoint
$f^*$. That is, $f^*$ is uniquely determined by the identity $\langle f(b), b' \rangle_{B_x B_s} = \langle b, f^*(b') \rangle_{B_y}$ for all $b \in B_y$ and $b' \in B_xB_s$.

 On
$\Hom(B_y, B_xB_s)$ we can consider the \emph{local
intersection form}
\[ (f, g)_y^{x,s} := g^* \circ f \in \End(B_y) = \RM. \]

\begin{thm}[Embedding theorem] \label{thm:embedding}
The map
\begin{gather*}
\iota : \Hom(B_y, B_xB_s) \to
(\ov{B_xB_s})^{-\ell(y)}  : f \mapsto \ov{f(c_\bot)}
\end{gather*}
is injective, with image contained in the primitive subspace
  \[ P^{-\ell(y)}_\rho \subset (\ov{B_xB_s})^{-\ell(y)}.\]Moreover,
  $\iota$ is an isometry with respect to the Lefschetz form up to a
  factor of $N$: for all
  $f, g \in \Hom(B_y, B_xB_s)$ we have
\begin{equation} \label{eq:scalarisometry}
N(f, g)_y^{x,s} = (\iota(f), \iota(g))^{-\ell(y)}_\rho.
\end{equation}
\end{thm}

\begin{remark} The above constructions depend on the choices ($\RM_{>0}$-torsors) of invariant forms on $B_y$ and $B_x$ and the choice (an $\RM^\times$-torsor) of $c_\bot \in
B_y^{-\ell(y)}$. The reader can confirm that both sides of
\eqref{eq:scalarisometry} are affected equally by any rescaling, and the coefficient of isometry $N$
is positive for any choice. \end{remark}


\begin{proof}
Consider the exact sequence of modules with $\Delta$-flag \[ \Delta_y = \Gamma_{\ge y} B_y \into B_y \onto B/\Gamma_{\ge y} B_y. \] We know that $\ch_\Delta \Delta_y = H_y$,
$\ch B_y = \un{H}_y$, and that $\ch_\Delta (B_y/\Gamma_{\ge y} B_y) =
\un{H}_y - H_y$ because this is part of the $\Delta$-flag on
$B_y$. Therefore the characters add up, and we can use Soergel's hom
formula (Theorem \ref{thm:hom}) to conclude that we have an exact sequence \[ \Hom^\bullet(B/\Gamma_{\ge y}B_y,
B_xB_s) \into \Hom^\bullet(B_y, B_xB_s) \onto \Hom^\bullet(\Delta_y, B_xB_s). \] Now $\un{H}_y - H_y \in \bigoplus v\ZM_{\ge 0}[v] H_z$ and $\ch(B_xB_s) =
\un{H}_x \un{H}_s \in \bigoplus \ZM_{\ge 0}[v] H_z$. Hence $\Hom^{\le 0}(B/\Gamma_{\ge y}B_y, B_xB_s) = 0$ and we have an isomorphism \[ \Hom(B_y, B_xB_s) \simto \Hom(\Delta_y,
B_xB_s) = \Gamma_y(B_xB_s)(\ell(y)). \]
Using Soergel's hom formula again we see that
$\Hom^\bullet(B_y,B_xB_s)$ is concentrated in degrees $\ge 0$ and
hence $\Gamma_y(B_xB_s)$ is concentrated in degrees 
$\ge \ell(y)$.
Now $B_xB_s$ is free as a right $R$-module and it is
known that $\Gamma_y(B_xB_s)$ is a direct 
summand of $B_xB_s$ as a right $R$-module (see the proof of
Proposition 6.4 in \cite{S3}). It follows that if $m \in B_xB_s$ and
$mr \in \Gamma_y(B_xB_s)$ for some $r \in R$ then $m \in
\Gamma_y(B_xB_s)$.
Hence the induced map
\[
\Gamma_y(B_xB_s)^{\ell(y)} \to (\overline{B_xB_s})^{\ell(y)}
\]
is injective.

Let $c$ be the image of a generator of $\Delta_y$ under $\Delta_y
\into \Gamma_y B_y \subset B_y$. It projects to a generator $\ov{c}$ of the
one-dimensional space $(\overline{B_y})^{\ell(y)} \cong \RM$. 
The isomorphisms of the previous paragraph imply that \[ \iota' : \Hom(B_y, B_xB_s) \to
\ov{B_xB_s}^{\ell(y)} : f \mapsto \ov{f(c)}. \] is injective. In
addition, $hL(y)$ implies that $\rho^{\ell(y)} \cdot c_\bot$ also has nonzero image in
$(\ov{B_y})^{\ell(y)}$, and therefore is equal to $\ov{c}$ up to a non-zero scalar. Hence \[ \iota : \Hom(B_y, B_xB_s) \to \ov{B_xB_s}^{-\ell(y)} : f \mapsto \ov{f(c_\bot)}
\] is injective too. Finally, $\rho^{\ell(y) + 1}$ annihilates $\ov{B_y}$ and hence the image of $\iota$ is contained in the primitive subspace $P^{-\ell(y)}_\rho \subset
(\ov{B_xB_s})^{-\ell(y)}$. The first part of the theorem now follows.

Fix $f, g \in \Hom(B_y, B_xB_s)$. We have
\begin{align*}
N(f,g)_y^{x,s} & = \langle g^* ( f (c_\bot)), \rho^{\ell(y)} \cdot c_\bot \rangle_{\ov{B_y}} \\ & =
\langle f (c_\bot), \rho^{\ell(y)} \cdot  g(c_\bot) \rangle_{\ov{B_x B_s}} \\ & = (\iota(f), \iota(g))_\rho^{-\ell(y)} \end{align*} (the first equality follows from
\eqref{eq:ynorm}, the second by adjointness, and the third by
definition). \eqref{eq:scalarisometry} now follows.\end{proof}

Because the restriction of a definite form to a subspace stays non-degenerate, we have

\begin{cor} $HR(x,s)$ and $S(y)$ for all $y < xs$ implies $S(xs)$. \end{cor}


\section{Hodge-Riemann bilinear relations} \label{sec:HR}

In this section we prove \eqref{eq:limit} from the introduction. We actually prove a more general version. Let us fix a (not necessarily reduced) expression $\un{x}$ and a
summand $B \summand BS(\un{x})$. On $B$ we have an invariant form induced from the intersection form on $BS(\un{x})$ and a Lefschetz operator induced by left multiplication by $\rho$.
Using the terminology of \S\ref{sec:lla}, for all $i \ge 0$ we get a Lefschetz form on
$(\overline{B})^{-i}$ given by
\[ (p,q)^{-i}_\rho = \Tr_\RM(\rho^i (pq)). \]

For all $\z \ge 0$ we consider the Lefschetz operator \[ L_\z := ( \rho \cdot - ) + \id_{B} ( \z \rho \cdot -) \] on $BB_s$. Here $ ( \rho \cdot - )$ denotes the operator of left
multiplication by $\rho$ and $\id_{B} ( \z \rho \cdot -)$ denotes the tensor product of the identity on $B$ and the operator of left multiplication by $\z\rho$ on $B_s$. In
this section $(-,-)^{-i}_\rho$ will always refer to the Lefschetz form on $\ov{B}$, while $(-,-)^{-i}_{L_\z}$ will refer to the Lefschetz form on $\ov{BB_s}$. Thus, $(-,-)^{-i}_{L_0}$ is
the Lefschetz form on $\ov{BB_s}$ induced by left multiplication by $\rho$. We abusively write $\Tr_\RM$ for the real valued trace on both $BS(\un{x})$ and $BS(\un{x}s)$.

\begin{thm} \label{thm:towards infinity} Suppose that $\ov{B}$
  satisfies hard Lefschetz and the Hodge-Riemann bilinear relations
  with the standard sign. Then for $\z \gg 0$, the induced action of $L_\z$ on
$\ov{BB_s}$ satisfies the hard Lefschetz theorem and the Hodge-Riemann
bilinear relations with the standard sign.\end{thm}

The following lemma reduces this theorem to a statement relating the signatures of the forms on $\ov{B}$ and $\ov{B B_s}$:

\begin{lem} \label{lem:hrsig} Let $V$ and $W$ be two finite
  dimensional graded vector spaces, equipped with graded non-degenerate
  symmetric forms and Lefschetz operators satisfying the hard
  Lefschetz theorem. Assume that $W$ is even or odd and that $\gdim V
  = (v+v^{-1}) \gdim W$, so that $V$ is odd or even. Suppose that $W$ satisfies the Hodge-Riemann bilinear relations with the standard sign.
Then $V$ satisfies the Hodge-Riemann bilinear
relations with the standard sign if and only if for all $i \ge 0$ the signature of the
Lefschetz form on the primitive subspace $P^{-i+1} \subset W^{-i+1}$ is
equal to the signature of the Lefschetz form on all of
$V^{-i}$. (By convention $P^1  = 0$.)\end{lem}

\begin{proof} Let $\ell \ge 0$ be such that $W^{-\ell}$ is the lowest
  non-zero degree of $W$. For $j \in \ZM$, write $v_j := \dim V^j$ and
  $w_j := \dim W^j$ for the Betti numbers of $V$ and $W$. For $j \ge
  0$ write $p_{-j} := v_{-j} - v_{-j-2}$ for the dimension of the
  primitive subspace $P^{-j} \subset V^{-j}$. Because  $\gdim V =
  (v+v^{-1}) \gdim W$ the lowest non-zero degree of $V$ is $-\ell - 1$
  and we have $v_{-j} = w_{-j + 1} + w_{-j-1}$. Hence, for all $j \ge 0$ we have
\[
p_{-j} = w_{-j+1}-w_{-j-3}.
\]
Now $V$ satisfies the
  Hodge-Riemann bilinear relations with the standard sign 
if and only if, for all $j \ge -1$
  the signature of the Lefschetz form on $V^{-j-1}$ is equal to
\begin{gather*}
(-1)^{(j+1-(\ell+1))/2}(p_{-j-1} - p_{-j-3} + p_{-j-5} - p_{-j-7} +
\dots ) = \\
= (-1)^{(j-\ell)/2}( (w_{-j} - w_{-j-4}) -  (w_{-j-2} - w_{-j-6}) +
(w_{-j-4} - w_{-j-8}) - \dots  )\\
= (-1)^{(j-\ell)/2}(w_{-j} - w_{-j-2}).
\end{gather*}
The last term is the signature of the Lefschetz form on the
primitive subspace $P^{-j} \subset W^{-j}$ by the Hodge-Riemann
bilinear relations. The lemma now follows.
\end{proof}


Clearly, the Lemma will apply to $W = \ov{B}$ and $V = \ov{BB_s}$, so long as $V$ satisfies hard Lefschetz. The
proof below establishes a statement about signatures. The essential argument is to show that, as $\z \to \infty$, the form on $\ov{BB_s}$ tends to the ``product" of the forms on $\ov{B}$ and on $\ov{B_s}$.

\begin{proof}[Proof of Theorem \ref{thm:towards infinity}] 
Recall from \S\ref{sec:invariant forms} the maps $\a$ and $\b$ from $B$ to $B B_s$, and the formulae
\eqref{alphaalpha}, \eqref{alphabeta} and \eqref{betabeta} which control the invariant form on $B B_s$. As a reminder, for $x \in B^{{-i}+1}$ and $y \in B^{{-i}-1}$ we have \[ \a(x) :=
xc_{\id} \quad \text{and} \quad \b(y) := yc_s \] in $(BB_s)^{-i}$.

We are interested in the $\RM$-valued form on $\ov{BB_s}$. It is immediate from \eqref{betabeta} that two elements in the image of $\b$ are orthogonal with respect to $\langle -,-
\rangle_{\ov{BB_s}}$, because the positive degree polynomial $\a_s$ appears on the right. For similar reasons, $L_\z = L_0$ when applied to an element in the image of $\b$, because left and
right multiplication by $\z \rho$ agree on $c_s \in B_s$. Therefore
\begin{equation} \label{eq:betabeta} ( \beta(y), \beta(y') )^{-i}_{L_\z} = 0 \end{equation}
and
\begin{gather} \label{eq:alphabeta}
\begin{array}{rl}
	( \alpha(x), \beta(y) )^{-i}_{L_\z} & = ( \alpha(x), \beta(y) )^{-i}_{L_0}
	\\ & = \Tr_{\RM}(\rho^i (xy) c_s) \\ & = ( x, \rho y)^{-i+1}_\rho
\end{array}
\end{gather}
This second equation relates the form on $(\ov{BB_s})^{-i}$ to the
form on $(\ov{B})^{-i+1}$. The only ``difficult'' pairings are of the
form $( \a(x), \a(x') )^{-i}_{L_\z}$. We will have more to say about
these below. 

Now fix $i \ge 0$ and choose elements $e_1, \ldots, e_n \in B^{-i-1}$ which project to an orthogonal basis of $(\ov{B})^{-i-1}$. Choose elements $p_1,\ldots,p_m \in B^{-i+1}$ which project to an orthogonal basis
of the primitive subspace $P^{-i+1}_\rho \subset (\ov{B})^{-i+1}$. Then \[\rho e_1, \dots, \rho e_n,p_1, \dots, p_m \] project to an orthogonal basis for $(\ov{B})^{{-i}+1}$. It follows that \[
\alpha(\rho e_1), \dots, \alpha(\rho e_n), \beta(e_1) \dots, \beta(e_n), \alpha(p_1), \dots, \alpha(p_m) \] project to a basis of $(\ov{BB_s})^{-i}$.

With this choice of basis, equations \eqref{eq:betabeta} and \eqref{eq:alphabeta} imply that the
Gram matrix of the form $( -, - )^{-i}_{L_\z}$ has the form
\[
M^{-i}_\z := \left ( \begin{array}{ccc}
* & J & * \\
J & 0 & 0 \\
* & 0 & Q_\z \end{array} \right )
\]
where $J$ is a non-degenerate diagonal matrix. We have not yet
computed $Q_\z$ or the $*$'s. The determinant of $M^{-i}_\z$ only
depends on the entries of $J$ and $Q_\z$. Hence $M^{-i}_\z$ is
non-degenerate if and only if 
$Q_\z$ is, in which case we can find a path in the space of
real non-degenerate symmetric matrices to the matrix
\[
M := \left ( \begin{array}{ccc}
0 & J & 0 \\
J & 0 & 0 \\
0 & 0 & Q_\z \end{array} \right )
\]
and we can conclude that the signature of $M^{-i}_\z$ is equal to the
signature of $Q_\z$.\footnote{More formally, let $Sym_n^{\det \ne 0}$
  denote the space of real non-degenerate symmetric matrices, with its
  Euclidean topology. We can find a path $t : [0,1] \to
  Sym_n^{\det \ne 0}$ such that $t(1) = M^{-i}_\z$ and $t(0) = M$. Using
  that the signature is constant on connected components of
  $Sym_n^{\det \ne 0}$ we conclude that the signatures of $M$ and
  $M^{-i}_\z$ coincide. Finally, the signatures of $M$ and $Q_\z$ are
  easily seen to agree.}

We claim that, for $\z \gg 0$, $Q_\z$ is non-degenerate and has
signature equal to the signature of $(-,-)^{-i+1}_\rho$ on
$P^{-i+1}_\rho \subset \ov{B}^{-i+1}$. If this is true then $L_{\z}$
satisfies hard Lefschetz, and Lemma \ref{lem:hrsig} will conclude the proof.

Firstly, if $i = 0$ then $m = 0$ and the result follows. Hence we may
assume $i > 0$. Let $p, q \in B^{-i+1}$. We have
\begin{align*}
(\a(p),\a(q))^{-i}_{L_\z} &= \Tr_{\RM}(L_\z^i((pq)c_\id))=\Tr_{\RM} \left
  (\sum_{j=0}^i{i \choose j} \rho^{i-j} (pq) (\z\rho)^jc_\id \right)
\end{align*}
By \eqref{eq:demazure} we have for $j \ge 1$
\[
\rho^{i-j} (pq) (\z\rho)^jc_\id  = \rho^{i-j} (pq) c_s
\cdot \partial_s((\z\rho)^j) + \rho^{i-j} (pq) c_{\id} \cdot s(\zeta\rho)^j.
\]
Applying $\Tr_{\RM}$ we obtain (again for $j \ge 1$)
\[
\Tr_{\RM}(\rho^{i-j} (pq) (\z\rho)^jc_\id) = \begin{cases}
\z \rho(\alpha_s^\vee) \Tr_{\RM}(\rho^{i-1} (pq)) & \text{if $j = 1$} \\
0 & \text{otherwise.}
\end{cases}
\]
Hence
\[
(\a(p),\a(q))^{-i}_{L_\z} = \Tr_{\RM}(\rho^i(pq)c_\id) + \z i\rho(\alpha_s^\vee) (p,q)_\rho^{-i+1}.
\]
Note that the first term is independent of $\z$. It follows that
\[ \lim_{\z \to \infty} \frac{1}{\z} Q_\z = i \rho(\a_s^\vee) \cdot Q \]
where $Q$ is the matrix $( (p_i,p_j)_\rho^{-i+1} )_{1 \le i, j \le n}$.
Now, $\ov{B}$ satisfies the Hodge-Riemann bilinear relations, and hence $Q$ is definite. It follows that $Q_\z$ is too, for $\z \gg 0$, and has the same
signature as $Q$ ($i$, $\z$ and $\rho(\a_s^\vee)$ are all strictly positive). The theorem now follows. \end{proof}

{The upshot of Theorem \ref{thm:towards infinity} is the following corollary.

\begin{cor} If $HR(\un{x})$ holds then $hL(\un{x},s)_\z$ for all $\z \ge 0$ implies
  $HR(\un{x},s)_\z$ for all $\z \ge 0$. \end{cor}
	
\begin{proof} By Theorem \ref{thm:towards infinity}, we have
  $HR(\un{x},s)_\z$ for some $\z \gg 0$. By Lemma \ref{lem:family} we
  have $HR(\un{x},s)_\z$ for all $\z \ge 0$. \end{proof}

All that remains is to prove hard Lefschetz for the family $L_\z$ of
Lefschetz operators. This task occupies the rest of the paper.


\section{Hard Lefschetz for Soergel bimodules} \label{sec:HL}

In this section we establish the hard Lefschetz theorem for Soergel bimodules using Rouquier complexes. Although the basic idea is simple, the details are somewhat complicated. Before
giving the details we give a brief motivational sketch:

Let us first recall a key fact from Hodge theory: the weak Lefschetz theorem together with the Hodge-Riemann bilinear relations in dimension $n-1$ imply the hard Lefschetz theorem in
dimension $n$. Let $X$ denote a smooth projective variety, and $X_H$ a
general hyperplane section and $r : X_H \hookrightarrow X$ the inclusion. A key point in the proof is the observation that one can factor the Lefschetz operator as the composition of the restriction map $r^* : H^*(X) \to H^*(X_H)$ and its dual $r_* : H^*(X_H)
\to H^{*+2}(X)$. The weak Lefschetz theorem implies that $r^*$ is injective in degrees $\le \dim_\CM X - 1$ and one can then use Lemma \ref{lem: weak lefschetz substitute} to deduce the
hard Lefschetz theorem for $H^*(X)$ from the Hodge-Riemann bilinear
relations for $H^*(X_H)$.

The weak Lefschetz theorem actually gives a situation stronger than that of Lemma \ref{lem: weak lefschetz substitute} because $r^*$ (resp. $r_*$) is an isomorphism in degrees $\le
\dim_\CM X - 2$ (resp. $\ge \dim_{\CM} X$), which can be used to deduce the Hodge-Riemann bilinear relations for $H^*(X)$ in all degrees except $\dim_\CM X$. This aspect of
the proof is not replicated in this paper.

A major initial hurdle in the setting of Soergel bimodules is the
apparent absence of the weak Lefschetz theorem. Indeed, even if geometric tools are available, taking a general
hyperplane section in a Bott-Samelson resolution or flag variety leaves the world of
varieties whose cohomology admits a simple combinatorial description.

The first key observation is that for any expression $\un{x}$,
left multiplication by $\rho$ on $\ov{BS(\un{x})}$ (our substitute for a Lefschetz operator) still admits a factorization (see \S\ref{sec:fact-lefsch-oper}) \[ \ov{BS(\un{x})}
\stackrel{\phi}{\longto} \bigoplus \ov{BS(\un{x}_{\hat{i}})}(1) \stackrel{\chi}{\longto} \ov{BS(\un{x})}(2). \] However, the modules appearing above will generally not satisfy hard
Lefschetz, because one has no control over the shifts of indecomposable Soergel bimodules that may occur.

The second key observation is that $\phi$ (and $\chi$) are (up to some positive scalars) differentials on Rouquier complexes. One can then use homological algebra to replace
$\ov{BS(\un{x})} \stackrel{\phi}{\longto} \bigoplus
\ov{BS(\un{x}_{\hat{i}})}(1) \to \dots$ by a minimal subcomplex without affecting exactness properties. It is this subcomplex that serves as a
replacement for the weak Lefschetz theorem, and allows us to deduce hard Lefschetz.

\subsection{Complexes and their minimal complexes} \label{sec:minimal complexes}
Let $C^b(\BC)$
denote the category of bounded complexes of Soergel bimodules (all differentials are required to be of degree zero) and let
$K^b(\BC)$ denote its homotopy category. Because we already use right
indices to indicate the degree in the grading we use left indices to
indicate the cohomological degree. In other words, an object $F \in
C^b(\BC)$ looks like
\[
\cdots \to {}^iF \stackrel{d}{\longto} {}^{i+1}F \to \cdots
\]
with ${}^iF \in \BC$ and $d$ a morphism in $\BC$. We regard $\BC$ as a full subcategory
of $C^b(\BC)$ and $K^b(\BC)$ consisting of complexes concentrated in
degree 0. As with bimodules we write $F' \summand F$ to mean that $F'$
is a direct summand (as complexes) of the complex $F$.

Let $\rad(\BC) \subset\BC$ denote the the radical of $\BC$ (see e.g.
\cite[\S1.8]{Kr}). It is an ideal of the category $\BC$ and we write
$\rad(\BC)(B, B') \subset \Hom(B,B')$ for the corresponding subspace,
for any $B, B' \in \BC$. On may
show \cite[Proposition 1.8.1]{Kr} that $\rad(\BC)(B,B)
\subset \End(B)$ coincides with the Jacobson radical
$J \End(B) \subset \End(B)$, for any $B \in \BC$.
Because the endomorphism ring of $B$ is a finite
dimensional $\RM$-algebra (remember that morphisms in $\BC$ are
assumed to be of degree zero), $\End B / J \End B$ is a semi-simple
$\RM$-algebra. We conclude that $\BC^{ss}
:= \BC/\rad(\BC)$ is semi-simple.

Given any indecomposable Soergel bimodule $B_x$, $\End B_x / J \End
B_x = \RM$. (For general reasons $\End B_x / J \End
B_x$ is a division algebra over $\RM$. However one always has a
surjection $\End B_x \onto \RM$ (see the proof of \cite[Satz
6.14]{S3}), and hence $\End B_x / J \End
B_x = \RM$.)  We conclude that the images of $B_x(i)$ for $x \in W$ and $i \in \ZM$
in $\BC^{ss}$ give a complete set of pairwise
non-isomorphic simple objects, all of whose endomorphism rings are isomorphic to $\RM$. We denote by $q : \BC \to \BC^{ss}$ the quotient functor. One may check that $f : B \to B'$ is an
isomorphism if and only if $q(f)$ is.

Now consider a complex $F \in C^b(\BC)$. We say that $F$ is
\emph{minimal} if all differentials on $q(F) \in C^b(\BC^{ss})$ are
zero. This is equivalent to requiring that 
$q(F)$ contain no contractible direct summands, which by the
isomorphism lifting statement, is equivalent to requiring that $F$
itself has no contractible direct summands.\footnote{If $F \in C^b(\BC)$ is a
complex, and a differential $d : {}^iF = M \oplus B \to {}^{i+1}F = M'
\oplus B' $ has the form
\[
\left ( \begin{matrix} \a & \b \\ \g & iso \end{matrix} \right ).
\]
for some isomorphism $iso : B \to B'$ then one can choose new decompositions ${}^iF = M \oplus B$ and ${}^{i+1}F = M' \oplus B'$
such that $d$ is a diagonal matrix, with entries $\a' : M \to M'$ and
$iso : B \to B'$. Hence $F$ is homotopic to a complex $F'$ with the
contractible summand $B \simto B'$ removed.
}

Given any complex $F \in C^b(\BC)$ there exists a direct summand $F_{\min}
\summand F$ such that $F_{\min}$ is minimal and such that the
inclusion $F_{\min} \to F$ is an isomorphism in $K^b(\BC)$. We call
such a summand a \emph{minimal subcomplex}.
Any two minimal subcomplexes are isomorphic as
complexes. (If $f : F \to G$ is a homotopy isomorphism
between minimal complexes then $q(f)$ is a homotopy isomorphism 
between complexes with trivial differential. It follows that $q(f)$,
and hence $f$, is an isomorphism
 of complexes.)

\subsection{The perverse filtration on bimodules}

A Soergel bimodule $B$ is \emph{perverse} if $\ch(B) = \sum a_x
\un{H}_x$ with $a_x \in \ZM_{\ge 0}$. A Soergel bimodule $B$ is
\emph{$p$-split} if each indecomposable summand of $B$ is isomorphic
to $B'(m)$ for some $m \in \ZM$ and perverse Soergel bimodule $B'$. Any summand of
a perverse (resp. $p$-split) Soergel bimodule is perverse
(resp. $p$-split) (use that $\overline{\ch(B_x)} = \ch(B_x)$ for any
$x \in W$ and that the character of any Soergel bimodule is positive
in the standard basis).

If $B_1$ and $B_2$ are perverse then Soergel's hom formula
(Theorem \ref{thm:hom}) implies that \begin{equation} \label{eq:perverse vanishing} \Hom(B_1, B_2(-i)) = 0 \quad \text{for $i > 0$}. \end{equation}

Let $B$ be a $p$-split Soergel bimodule and choose a decomposition
\begin{equation}
  \label{eq:p-split decomp}
  B \simto \bigoplus B_x^{\oplus m_{x,i}}(i)
\end{equation}
of $B$ as a direct sum of indecomposable bimodules. Because $B$ is
assumed $p$-split, we know that if $m_{x,i} \ne 0$ then $B_x$ is
perverse. We define the \emph{perverse filtration} to be the
filtration
\[
\ptau_{\le j} B := \bigoplus_{i \ge -j} B_x^{\oplus m_{x,i}}(i).
\]
(The more geometrically minded reader might prefer
  ${}^p \tau_{\le j}$.) Using \eqref{eq:perverse vanishing} one can show that this
filtration does not depend on the choice of decomposition
\eqref{eq:p-split decomp} and is preserved (possibly non-strictly) by
all maps between Soergel bimodules. Of course this filtration always
splits, however the splitting is not canonical in general.

We set $\ptau_{< j} := \ptau_{\le j-1}$ and define
\[
\ptau_{\ge j} B := B/ \ptau_{<j} B
\]
and
\[
 \HC ^j(B) := \ptau_{\le j}(B)/\ptau_{< j}(B)(j).
\]
One can check $\ptau_{\le j}(-)$, $\ptau_{\ge j}(-)$ and $\HC^j(-)$
define endofunctors on the full subcategory of $p$-split Soergel
bimodules.

\subsection{The perverse filtration on complexes}
Let ${}^p K^b(\BC)^{\ge 0}$ denote the full subcategory of
$K^b(\BC)$ with objects those complexes which are isomorphic to
complexes $F$ such that
\begin{enumerate}
\item each term of $F$ is $p$-split;
\item $\ptau_{<-i}{}^iF = 0$ for all $i\in \ZM$.
\end{enumerate}
Similarly, we define ${}^p K^b(\BC)^{\le 0}$ to be the full
subcategory of complexes which are isomorphic to complexes $F$ such that
\begin{enumerate}
\item each term of $F$ is $p$-split;
\item ${}^iF = \ptau_{\le -i}{}^iF$ for all $i \in \ZM$.
\end{enumerate}
Alternatively, $F$ belongs to ${}^p K^b(\BC)^{\le 0}$
(resp. ${}^p K^b(\BC)^{\ge 0}$) if and only if its minimal complex
satisfies the conditions above.

In other words, if a minimal complex is in ${}^p K^b(\BC)^{\ge 0}$ then an indecomposable summand in cohomological degree $0$ has the form $B_x(k)$ for $k \le 0$, an
indecomposable summand in cohomological degree $1$ has the form $B_x(k)$ for $k \le 1$, etc.

\begin{lem} \label{lem:middle term} Let $F' \to F \to F'' \triright$ be a distinguished triangle in $K^b(\BC)$. If $F', F'' \in {}^pK^b(\BC)^{\ge 0}$
then $F \in {}^pK^b(\BC)^{\ge 0}$. Similarly, if $F', F'' \in {}^pK^b(\BC)^{\le 0}$ then $F \in {}^pK^b(\BC)^{\le 0}$. \end{lem}

\begin{proof} We prove the first statement; the second statement follows by an identical argument. We may assume that ${}^iF'$ and ${}^iF''$ are $p$-split and that $\ptau_{<
-i}{}^iF' = \ptau_{<-i}{}^iF'' =0$ for all $i \in \ZM$. Turning the triangle we see that $F$ is isomorphic to the cone over a map $F''[-1] \to F'$. This cone has $i^{th}$
term ${}^iF'' \oplus {}^iF$. The result follows because
$\ptau_{<-i}({}^iF'' \oplus {}^iF') = 0$ for all $i \in \ZM$. \end{proof}

\begin{remark} Once one has proven Soergel's conjecture one may show
  that (${}^p K^b(\BC)^{\le 0}$, ${}^p K^b(\BC)^{\ge 0}$) gives a
  non-degenerate $t$-structure on $K^b(\BC)$. Its
heart can be thought of as a category of mixed equivariant perverse sheaves on the (possibly non-existent) flag variety associated to $(W,S)$. \end{remark}

\subsection{Rouquier complexes}

The monoidal structure on $\BC$ induces a monoidal structure on
$K^b(\BC)$ (total complex of tensor product of complexes) which we denote by juxtaposition.
Given a distinguished triangle $F' \to F \to F'' \triright$
and $G \in K^b(\BC)$ the triangle \[ F' G \to F G \to F'' G \triright \] is also distinguished.

For $s \in S$ consider the complex \[ F_s:= 0 \to B_s \to R(1) \to 0 \] where $B_s$ occurs in cohomological degree 0 and the only non-zero differential is given by the
multiplication map $f \ot g \mapsto fg$. It is known and easily
checked that $F_s$ is invertible in $K^b(\BC)$, hence tensoring on the
left or right by $F_s$ gives an equivalence of
$K^b(\BC)$.

Fix $x \in W$ and a reduced expression $\underline{x} = s_1s_2 \cdots s_m$. As an object in the homotopy category $K^b(\BC)$, the object $F_{s_1} \cdots F_{s_m}$ depends only on $x$ up to
canonical isomorphism (see \cite{Ro}). In this paper, a \emph{Rouquier complex} is any choice $F_x \summand F_{s_1} \cdots F_{s_m}$ of minimal subcomplex, which
again does not depend on the choice of reduced expression.

\begin{remark}
  Braid group actions on derived categories appearing in highest
  weight representation theory have 
  been around for decades (see e.g. \cite{Ca,Ri}). One obtains the
  above complexes by translating these actions into Soergel
  bimodules. The term ``Rouquier complex'' seems to have been
  introduced by Khovanov. We feel it is justified in our setting, because it was
  Rouquier who first emphasised that concrete algebraic properties of
  these complexes should have applications for arbitrary Coxeter
  systems \cite[4.2.1]{Ro2}. This is a key idea in the present article.
\end{remark}

A straightforward induction shows that $F_x$ is homotopic to
$R(-\ell(x))$ when viewed as a complex of right $R$-modules. This
implies the following lemma:

\begin{lem} \label{lem:Rouqcohom}
  We have
\[
H^i(\overline{F_x}) = \begin{cases} \RM(-\ell(x)) & \text{if $i = 0$}
  \\ 0 & \text{otherwise}. \end{cases}
\]
\end{lem}

For the rest of this section and the next we examine the perverse
filtration on Rouquier complexes.

\begin{lem} \label{lem:fs} Let $x \in W$, $s \in S$ and assume
  $S(x)$. Regard $B_x \in K^b(\BC)$ as a complex concentrated in degree $0$.
\begin{enumerate} \item If $xs < x$ then $B_xF_s \cong B_x(-1)$ in
  $K^b(\BC)$.
\item If $xs > x$ then $B_xF_s \in {}^pK^b(\BC)^{\ge 0}.$
\end{enumerate} \end{lem}

\begin{proof} (1) Under our assumptions $B_xB_s \cong B_x(1) \oplus B_x(-1)$. Hence $B_xF_s$ has the form \[ 0 \to B_x(1) \oplus B_x(-1) \to B_x(1) \to 0. \] Now $B_x$ is indecomposable
and tensoring with $F_s$ gives an equivalence of $K^b(\BC)$. Hence the above complex is also indecomposable. It follows that the map $B_x(1) \to B_x(1)$ induced by the differential is
non-zero, and is an isomorphism because $\End(B_x)=\RM$. It follows that the subcomplex $B_x(1) \to B_x(1)$ is contractible, yielding the result.

(2) If Soergel's conjecture holds for $B_x$ then
$\ch(B_xB_s) = \un{H}_x\un{H}_s \in \bigoplus \ZM_{\ge 0} \un{H}_z$
and $B_xB_s$ is perverse. The result is now immediate from the definitions.\end{proof}

\begin{lem} Suppose that $F \in {}^pK^b(\BC)^{\ge 0}$ and that
  Soergel's conjecture holds for all indecomposable summands of all ${}^iF$. Then $F F_s\in {}^pK^b(\BC)^{\ge
0}$. \end{lem}

\begin{proof} We can assume that $F$ is a minimal complex.
  Consider the stupid filtration of $F$:
\[
w_{\ge k} F := \dots \to 0 \to  {}^kB \to {}^{k+1}B \to \dots 
\]
Then for all $k$ we have distinguished triangles
\[
w_{\ge k+1}F \to w_{\ge k} F \to {}^kB[-k] \triright
\]
By Lemma \ref{lem:middle term} if $(w_{\ge k+1}F)F_s$ and ${}^kB[-k]F_s$ are in
${}^pK^b(\BC)^{\ge 0}$ then so is $(w_{\ge k}B)F_s$. By Lemma
\ref{lem:fs}, ${}^kB[-k] F_s \in {}^pK^b(\BC)^{\ge 0}$. The result now
follows by induction.
\end{proof}

\begin{cor} \label{lem:p>0} Assume $S(y)$ for all $y<x$. Then $F_x \in {}^pK^b(\BC)^{\ge 0}.$ \end{cor}

\begin{proof} Choose a reduced expression $\un{x}$ for $x$, ending in some $s \in S$. Let $y = xs < x$. By an inductive application of the previous lemma, $F_y \in {}^pK^b(\BC)^{\ge 0}$.
Then $F_x \cong F_y F_s$, and $F_y$ satisfies the conditions of the previous lemma, so $F_x \in {}^pK^b(\BC)^{\ge 0}$. \end{proof}

In particular, in the setting of the above corollary, we know
Soergel's conjecture for every summand of every ${}^iF_x$ except possibly $B_x$
itself, which occurs only in degree zero.

\subsection{Rouquier complexes are linear}

In the present section we establish that Rouquier complexes are ``linear'' under the assumption of Soergel's conjecture. We will need the following result of Libedinsky and the second
author:

\begin{prop} \label{prop:delta split} (Rouquier complexes are
  $\Delta$-split) Fix $x \in W$ and let $F_x$ denote a Rouquier complex. Then for any $y \in W$ we have an
isomorphism in the homotopy category of graded $R$-bimodules \[
\Gamma_{\ge y / > y} F_x = \begin{cases} \Delta_x & \text{if $x = y$,}
  \\ 0 & \text{otherwise.} \end{cases} \] \end{prop}

\begin{proof} This is \cite[Proposition 3.7]{LW}. \end{proof}

The precise statement of ``linearity'' is the following:

\begin{thm} (Rouquier complexes are linear) \label{thm:rouquier
    linear} Assume $S(y)$ for all $y \le x$ and let
  $F_x$ denote a Rouquier (minimal) complex. We have
  \begin{enumerate}
  \item ${}^0F_x = B_x$;
  \item For $i \ge 1$, ${}^iF_x= \bigoplus B_z(i)^{\oplus m_{z,i}}$ for $z < x$ and $m_{z,i} \in \ZM_{\ge 0}$.
  \end{enumerate}
In particular, $F_x \in {}^p K^b(\BC)^{\le 0} \cap {}^p K^b(\BC)^{\ge 0}$.
\end{thm}

\begin{remark}(Positivity of inverse Kazhdan-Lusztig polynomials.) One
  can show that 
\[ H_x = \ch(F_x) := \sum (-1)^i \ch({}^iF_x). \]
Therefore, defining $g_{z,x}$ by $H_x =
\sum g_{z,x} \un{H}_z$, we have $g_{x,x} = 1$ and $g_{z,x} = \sum (-1)^i m_{z,i} v^i$ for $z \le x$. Hence one can determine all multiplicities $m_{z,i}$ using only Kazhdan-Lusztig
combinatorics. Furthermore, a straightforward inductive argument gives
that $m_{z,i} = 0$ if $i$ and $\ell(x) - \ell(z)$ have different parity. Hence $(-1)^{\ell(x)-\ell(z)}g_{z,x}$ has positive
coefficients for all $z \le x$.
 \end{remark}

The theorem will be deduced from the following:

\begin{lem} \label{lem:delta induction} Assume $S(\le\!x)$. If ${}^iF_x$ contains a summand
  isomorphic to $B_z(j)$ with $z < x$ then ${}^{i-1}F_x$ contains a summand isomorphic to $B_{z'}(j')$ with $z'>z$ and
$j' < j$.\end{lem}

In the proof we use the following facts which are immediate from the
definition of the $\Delta$-character: if $S(y)$ holds then $\Gamma_{\ge
z/>z}(B_y)$ is zero unless $y \ge z$; it is $\Delta_z$ when $y=z$; and it is $\oplus \Delta_z(k)^{\oplus m_k}$ for $y>z$ with all $k$ strictly positive. 

\begin{proof}
Choose a summand $B_z(j) \summand {}^iF_x$, and consider its image in
${}^{i+1}F_x$ under the differential. By \eqref{eq:perverse vanishing}
this image must project trivially to any summand of the form $B_y(k)$
for $k<j$. If $S(z)$ and $S(y)$ hold
then Soergel's hom formula (Theorem \ref{thm:hom}) implies that any non-zero map $B_z(j) \to
B_y(j)$ is an isomorphism (and so $z = y$). Such an isomorphism cannot
appear as the projection of the differential in a minimal complex,
because it would yield a contractible summand. Therefore, $B_z(j)$
maps to 
$\ptau_{<-j} {}^{i+1}F_x$, the sum of terms $B_y(k)$ for $k>j$. Similarly, if some summand $B_y(k)$ of ${}^{i-1}F_x$ is sent non-trivially to $B_z(j)$ by the differential and then projection,
we must have $k<j$.

Now apply $\Gamma_{\ge z/>z}$ to $F_x$. The result is split by Proposition \ref{prop:delta split}, and has a summand in $\Gamma_{\ge z/>z}{}^iF_x$ isomorphic to $\Delta_z(j)$ coming from
our chosen summand $B_z(j)$. This summand cannot survive in the cohomology of the complex,
and thus must map isomorphically to some $\Delta_z(j)$ in $\Gamma_{\ge z/>z}({}^{i+1}F_x)$, or be mapped to isomorphically from some $\Delta_z(j)$ in $\Gamma_{\ge z/>z}({}^{i-1}F_x)$. The former
is impossible, because this summand maps to $\Gamma_{\ge z/>z}
\ptau_{<-j} ({}^{i+1}F_x)$ which can only contain $\Delta_z(k)$ for
$k>j$. Thus some summand $B_y(k)$ of ${}^{i-1}F_x$ contributes
$\Delta_z(j)$ to $\Gamma_{\ge z/>z}$ (in particular $y \ge z$), and
this maps to $B_z(j)$. As mentioned above we must have $k<j$, which
means that $y > z$. This proves the lemma. \end{proof}  

\begin{proof}[Proof of Theorem \ref{thm:rouquier linear}.] Lemma \ref{lem:delta induction} implies that the only summands of ${}^0F_x$ are of the form $B_x$, because
$^{-1}F_x=0$. In fact, $^0F_x \cong B_x$, as can be seen by applying
$\Gamma_{\ge x/>x}$. Induction using Lemma \ref{lem:delta induction}
then implies that $\ptau_{> -i}{}^iF_x = 0$. The theorem now follows
because $F \in {}^p K^b(\BC)^{\le 0}$ by Lemma \ref{lem:p>0}. \end{proof}

\subsection{Rouquier complexes are Hodge-Riemann} \label{sec:RHR}
We will use the
following proposition repeatedly in what follows:

\begin{prop} \label{prop:sumHR=>HR}
Fix $\z \ge 0$, $s \in S$ and a Soergel bimodule $B =
  \bigoplus_{z \in W} B_z^{\oplus m_z}$ (for $m_z \in
    \ZM_{\ge 0}$) such that if $m_z \ne 0$ then $S(z)$
  and $HR(z,s)_\z$ hold. If $\z = 0$ we assume in addition that $m_z
  = 0$ if $zs < z$.

Assume that $\overline{B}$ is even or odd and
  that $B$
  is equipped with an invariant non-degenerate form $\langle -, -
  \rangle_B$ such that $\overline{B}$ satisfies the Hodge-Riemann
  bilinear relations with the standard sign (with respect to left
  multiplication by $\rho$ and $\langle
  -, - \rangle_{\overline{B}}$).

Then $\overline{BB_s}$ satisfies the
  Hodge-Riemann bilinear relations with the standard sign (with
  respect to $L_\z$ and the induced form $\langle - , -
  \rangle_{\overline{BB_s}}$).
\end{prop}

\begin{proof} We claim that we can choose our isomorphism $B \cong
  \bigoplus B_z^{\oplus m_z}$ such that each indecomposable summand is orthogonal
  under $\langle -, - \rangle_B$. Because Soergel's conjecture
  holds for each summand, the decomposition into isotypic components
  must be orthogonal, as $\Hom(B_z,\DM B_y)=0$ for $y \ne
  z$. Applying Soergel's conjecture again, we know that $\End(B_z^{\oplus
    m_z})$ is a matrix algebra, and choosing a decomposition of
  $B_z^{\oplus m_z}$ is the same as choosing a basis for $(B_z^{\oplus m_z})^{-\ell(z)}$. It is not difficult to check
  that if one chooses an orthogonal basis of
  $(B_z^{\oplus m_z})^{-\ell(z)}$ with respect to the (definite)
  Lefschetz form then one obtains an orthogonal decomposition of $B_z^{\oplus
    m_z}$.

Hence we may assume that the decomposition $B = \bigoplus B_z^{\oplus
  m_z}$ is orthogonal with respect to $\langle -, - \rangle$. It
follows that the induced form is orthogonal with respect to the
decomposition $BB_s = \bigoplus (B_zB_s)^{\oplus m_z}$. By the
Hodge-Riemann bilinear relations for $\overline{B}$, the
Lefschetz form on the primitive subspace in degree $m + 2i$ is
$(-1)^{i}$-definite, where $m$ denotes the minimal non-zero degree in
$\overline{B}$. Hence the restriction of $\langle -, - \rangle_{B}$ to
any summand isomorphic to $B_z$ is $(-1)^{(\ell(z) - m)/2}$ times a
positive multiple of the intersection form on $B_z$. It follows
from $HR(z,s)_{\z}$ that the Lefschetz form on the primitive subspace of
each summand $\overline{B_zB_s} \subset \overline{BB_s}$ is $(-1)^i$-definite in degree $-m-1 + 2i$. Hence the same is
true of $\overline{BB_s}$ (being the orthogonal direct sum of such
spaces). Hence $\overline{BB_s}$ satisfies the Hodge-Riemann bilinear
relations with the standard sign as claimed.
\end{proof}

For the rest of this section, fix $x \in W$ and assume $S(\le\!x)$. By Theorem \ref{thm:rouquier linear}
we know that ${}^jF_x$ is concentrated in perverse degree $-j$. 
By definition, $F_x$ is a direct summand of $F_{s_1} \cdots
F_{s_m}$ for any choice of reduced expression $\un{x} = s_1 \cdots
s_m$. Hence ${}^jF_x$ is a direct summand of ${}^j(F_{s_1} \cdots
F_{s_m})$. In other words, for all $j \ge 0$, 
\begin{equation} \label{eq:BSembedding}
 {}^jF_x \summand \bigoplus_{\un{x}' \in \pi(\un{x}, j)}
BS(\un{x}')(j)
\end{equation} where $\pi(\un{x}, j)$ denotes the set
of all subexpressions of $\un{x}$ obtained by 
omitting $j$ simple reflections. Shifting, we deduce that $
{}^jF_x(-j)$ is a summand of $\bigoplus BS(\un{x}')$.

Fix a tuple $\l = (\l_{\un{x}'})_{\un{x}' \in \pi(\un{x}, j)}$ of strictly
positive real numbers. We use these scalars to rescale the direct sum
of the intersection form on $\bigoplus BS(\un{x}')$: if $b =
(b_{\un{x}'})$ and $b' = (b_{\un{x}'})$ are elements of $\bigoplus
    BS(\un{x}')$ we set
\[
\langle b, b' \rangle^{\l} := \sum_{\un{x}' \in \pi(\un{x}, j)} \lambda_{\un{x}'} \langle b_{\un{x}'},
b'_{\un{x}'} \rangle_{BS(\un{x}')}.
\]

We say that $F_x$ \emph{satisfies the Hodge-Riemann bilinear
  relations} if for all reduced expressions $\un{x} = s_1 \dots s_m$
one can choose an embedding \[F_x \summand F_{s_1} \cdots
F_{s_m}\]such that, for all tuples of strictly positive real numbers $\l =
(\l_{\un{x}'})$, each $\overline{{}^jF_x(-j)}$ satisfies the
Hodge-Riemann bilinear 
relations with respect to the form induced by $\langle - , -
\rangle^\l$ and the Lefschetz operator given by left multiplication by
$\rho$, and with global sign determined as follows: the 
Lefschetz form should be positive definite on primitive subspaces in
degrees congruent to $-m + j$ modulo 4. (One can show that
this is equivalent to satisfying the Hodge-Riemann bilinear relations
with the standard sign. We will not need this.)

\begin{prop} \label{prop:signs} Assume $S(\le\!x)$. Also assume that $HR(y,s)$ holds for all $y
  < x$ and $s \in S$ with $ys > y$. Then $F_x$ satisfies the
  Hodge-Riemann bilinear relations.\end{prop}

\begin{proof} We prove the proposition by induction over the Bruhat
  order, with the case $x = \id$ being obvious.  Fix a reduced
  expression $\un{x} = s_1s_2 \dots s_m$ for $x$  as above and let
  $\un{y} = s_1 \dots s_{m-1}$ and $s = s_m$ so that $x = 
  ys$. By induction we may
assume that $F_{y}$ satisfies the Hodge-Riemann bilinear
relations. Hence we may choose an embedding $F_y \summand F_{s_1}
\cdots F_{s_{m-1}}$ such that for all $j$ and any choice of scalars
$(\mu_{\un{y}'})_{\un{y}' \in \pi(\un{y}, j)}$ the form on
  $\overline{{}^j F_y(-j)}$ induced by the pullback of the form
  $\langle -, - \rangle^{\mu}$ under the embedding
\[
{}^j F_y(-j) \summand \bigoplus_{\un{y}' \in \pi(\un{y}, j)} BS(\un{y}')
\]
satisfies the Hodge-Riemann bilinear relations. Now $F_x$ is a summand
of $F_yF_s$ and hence we have natural embeddings
\begin{align*}
{}^jF_x(-j) &\summand {}^jF_yB_s(-j) \oplus \; {}^{j-1}F_y(-j + 1)
\summand \\ & \quad 
\summand \bigoplus_{ \un{y}' \in \pi(\un{y}, j)} BS(\un{y}')B_s \oplus
\bigoplus_{ \un{y}'' \in \pi(\un{y}, j-1)} BS(\un{y}'') =
\bigoplus_{ \un{x}' \in \pi(\un{x}, j)} BS(\un{x}').
\end{align*}
We claim that $\overline{{}^j F_x(-j)}$ satisfies the Hodge-Riemann
bilinear relations with respect to this embedding, for any tuple $\l =
(\l_{\un{x}'})_{\un{x}' \in \pi(\un{x}, j)}$  of strictly positive
real numbers (or equivalently any pair $(\mu'_{\un{y}'})_{\un{y}' \in
  \pi(\un{y}, j)}$ and $(\mu''_{\un{y}'})_{\un{y}'' \in \pi(\un{y}, j-1)}$
  of tuples of strictly positive real numbers).

Soergel's conjecture holds for all indecomposable summands of $F_{y}$
and hence we have a
canonical decomposition
\[ {}^jF_{y}(-j) = \bigoplus_{z \in W} V_z
\otimes_{\RM} B_z\]
for some (degree zero) multiplicity spaces $V_z$. Set \[
B^{\uparrow} := \bigoplus_{z \in W \atop zs > z} V_z \otimes_{\RM} B_z 
\quad \text{and} \quad B^{\downarrow} := \bigoplus_{z \in W \atop zs <
  z} V_z \otimes_{\RM} B_z \] so that
\begin{equation} \label{eq:Fydecomp} {}^jF_{y}(-j) = 
B^{\uparrow} \oplus B^{\downarrow}.\end{equation}
This decomposition is
orthogonal with respect to the induced forms because
$\Hom(B^{\uparrow},  \DM B^{\downarrow}) = \Hom(B^{\downarrow} ,\DM B^{\uparrow}) = 0$. Character calculations yield that $B^{\uparrow}B_s$ is perverse, and that $B^{\downarrow}B_s \cong B^{\downarrow}(-1) \oplus
B^{\downarrow}(1)$ (see also the proof of Theorem \ref{thm:hLI}).

Now, as we have already remarked above, $F_x$ is a summand of
$F_{y}F_s$ and so ${}^jF_x$ 
is a summand of $^jF_{y}B_s \oplus {}^{j-1} F_{y}(1)$. We
rewrite this using \eqref{eq:Fydecomp}:
\[
^jF_x(-j) \summand B^{\downarrow}B_s \oplus B^{\uparrow}B_s \oplus {}^{j-1}F_{y}(-j+1).
\]
This decomposition is orthogonal with respect to the induced forms 
and the inclusion of ${}^jF_x(-j)$  is an isometry.

The decomposition $B^{\downarrow}B_s \cong B^{\downarrow}(-1) \oplus B^{\downarrow}(1)$ is not orthogonal with respect to the induced form. In fact,
the induced form is non-degenerate, and hence induces a non-degenerate
pairing of $B^{\downarrow}(-1)$ and $B^{\downarrow}(1)$. Nonetheless,
we claim that in the decomposition
\[
\overline{B^{\downarrow}B_s} \cong  \overline{B^{\downarrow}}(1)
\oplus \overline{B^{\downarrow}}(-1)
\]
the restriction of the Lefschetz form to $\overline{B^{\downarrow}}(1)$
is zero.  Indeed, our assumptions imply that left multiplication by
$\rho$ satisfies the hard Lefschetz theorem on
$\overline{B^{\downarrow}}$, and the fact that the Lefschetz form
is zero follows from Lemma \ref{lem:shiftedvanishing}.

Because ${}^j F_x(-j)$ lives in perverse degree $0$ by Theorem \ref{thm:rouquier
linear}, Hom vanishing \eqref{eq:perverse vanishing} implies that the projection ${}^j F_x(-j) \to B^{\downarrow}B_s$ will land entirely within
$B^{\downarrow}(1)$, and therefore the image of $\overline{{}^j F_x(-j)}$
in $\overline{B^{\downarrow}B_s}$ will not contribute to the Lefschetz form.

Hence the projection to the second two factors above gives a map \[ \iota : \overline{{}^jF_x(-j)} \to\overline{B^{\uparrow}B_s} \oplus
\overline{{}^{j-1}F_{y}(-j+1)} \] which is an isometry for the
Lefschetz forms. We claim that $\iota$ is injective. Recall the functor
$q : \BC \to \BC^{ss}$ from \S\ref{sec:minimal complexes}. Any map ${}^j F_x(-j) \to
B^{\downarrow}(1)$ cannot be an isomorphism because it is a map
between objects in perverse degrees $0$ and $-1$ respectively, and
hence vanishes after applying $q$. On the other hand, if we apply $q$
to the original inclusion ${}^jF_x \hookrightarrow {}^j(F_yF_s)$ then
we obtain an injection, because this map is the inclusion of a direct
summand. We conclude that if we apply $q$ to $\iota' : {}^jF_x(-j)  \to
B^{\uparrow}B_s \oplus {}^{j-1}F_{y}(-j+1)$ then we obtain a split
inclusion. Hence $\iota'$ is a split inclusion and the claim follows.

By assumption the induced intersection form on $\overline{^{j-1}F_{y}(-j+1)}$ satisfies the Hodge-Riemann bilinear relations, and is positive
definite on primitive subspaces in degrees congruent to $-(m-1)+j-1= -m+j$ modulo 4. By the same inductive assumption, $\overline{B^\uparrow}$
satisfies the Hodge-Riemann bilinear relations with global sign given as follows: the Lefschetz form is $>0$ on primitives in degrees congruent to
$-(m-1)+j$ modulo 4. By Proposition \ref{prop:sumHR=>HR} (essentially applying $HR(z,s)$ to each summand) it follows that $\ov{B^\uparrow B_s}$
satisfies the Hodge-Riemann bilinear relations, with the Lefschetz forms $>0$ on primitive subspaces in degrees congruent to $-(m+1)+j-1 = -m+j$
modulo 4.

We conclude that the codomain of $\iota$ satisfies the Hodge-Riemann
bilinear relations. Hence the same is true for 
$\overline{{}^jF_x(-j)}$,  being a $\rho$-stable summand with symmetric
Betti numbers (Lemma \ref{lem: HR subspace}).
\end{proof}


\subsection{Factoring the Lefschetz operator} \label{sec:fact-lefsch-oper}

Fix an expression $\un{x} = s_1s_2 \cdots s_m$. Recall the
morphisms $\Br_i$, $\phi_i$ and $\chi_i$ introduced in
\S\ref{sec:BSbimodules}. Let us denote by $\Tr$ and $\langle -, -
\rangle$ the trace 
map and intersection form on $BS(\un{x})$. To avoid confusion, we
denote the trace map and intersection form on $BS(\un{x}_{\widehat{i}})$
by $\Tr_i$ and $\langle -, - \rangle_i$.

\begin{lem} \label{lem: breaking lefschetz II}
  For $b, b' \in BS(\un{x})$ we have $\langle b, \Br_i b' \rangle = \langle \phi_i b, \phi_i b' \rangle_i.$
\end{lem}

\begin{proof}
  We may assume $b = b_1b_2 \cdots b_m$ and $b' = b_1'b_2' \cdots b_m'$
  with $b_i, b_i' \in B_{s_i}$. We calculate
\begin{align*}
\langle b, \Br_i b' \rangle &= \Tr((b_1b_1') \cdots (b_ib_i'c_s) \cdots
(b_mb_m')) \\
& = \Tr((b_1b_1') \cdots \mu(b_i)\mu(b_i')c_s \cdots (b_mb_m')) \\
& = \Tr( \chi_i( (b_1b_1') \cdots \mu(b_i)\mu(b_i') \cdots (b_mb_m'))) \\
& = \Tr_i((b_1b_1') \cdots \mu(b_i)\mu(b_i') \cdots (b_mb_m')) \\
& = \langle \phi_i(b), \phi_i(b') \rangle_i.
\end{align*}
The second to last equality follows from the identity $\Tr(\chi_i(\g)) =
\Tr_i(\g)$ valid for all $\g \in BS(\un{x}_{\widehat{i}})$.
\end{proof}

Let us rescale the forms on each $BS(\un{x}_{\widehat{i}})$ by defining
\[
\langle -, - \rangle'_i := ( s_{i-1} \cdots s_1 \rho)(\a_{s_i}^\vee) \langle -, - \rangle_i.
\]
Let $\langle -, - \rangle'$ denote the direct sum of the forms
$\langle -, - \rangle'_i$ on $\bigoplus BS(\un{x}_{\widehat{i}})$.

If we set $\phi := \sum \phi_i$ then, for $b, b' \in BS(\un{x})$
we have
\begin{align*}
  \langle \phi(b), \phi(b') \rangle' & = \sum_{1 \le i \le m} (s_{i-1}
  \cdots s_1 \rho)( \a_{s_i}^\vee) \langle \phi_i(b), \phi_i(b') \rangle_i \\
& = \sum_{1 \le i \le m} ( s_{i-1} \cdots s_1 \rho )(\a_{s_i}^\vee) \langle b, \Br_i(b') \rangle_i \\
& = \langle b, \rho b' \rangle - \langle b, b' \rangle \cdot w^{-1}\rho
\end{align*}
by Lemmas \ref{lem: breaking lefschetz} and \ref{lem: breaking
  lefschetz II} respectively. We conclude:

\begin{lem} \label{lem:factored leftschetz}
  Consider the induced map
\[
\ov{BS(\un{x})} \stackrel{\phi}{\longto} \bigoplus \ov{BS(\un{x}_{\widehat{i}})}(1)
\]
For all $b, b' \in \ov{BS(\un{x})}$ we have
\[
\langle b, \rho b' \rangle = \langle \phi b, \phi b' \rangle' \in \RM.
\]
\end{lem}

\begin{remark}
  Lemma \ref{lem:factored leftschetz} will be a key tool in our proof
  of the hard Lefschetz theorem for Soergel bimodules. It serves as
  a partial replacement for the weak Lefschetz theorem.
\end{remark}

\begin{remark} When we apply the above lemma, $\un{x}$ will be a
  reduced expression. Because $\rho$ is assumed dominant regular it
  follows that all the scaling factors $(s_{i-1} \cdots s_1
  \rho)(\a_{s_i}^\vee)$ 
are positive, by \eqref{eq:pos}. Hence, although we rescale the forms on each
Bott-Samelson bimodule, this does not affect the signs appearing in the
Hodge-Riemann bilinear relations.
\end{remark}

\subsection{Proof of hard Lefschetz}
Fix $x \in W$ and $s \in S$. Let $\un{x}$
denote a reduced expression for $x$. Recall the operator $L_\z$ on
$B_xB_s$ from \S\ref{sec:HR}. The goal of this section is to prove
three incarnations of the hard Lefschetz theorem for the induced
action of $L_\z$ on $\overline{B_xB_s}$ under certain inductive
  assumptions. The three cases are:
  \begin{enumerate}
  \item $\z > 0$ and $xs < x$ (Theorem \ref{thm:hLI}),
  \item $\z > 0$ and $xs > x$  (Theorem \ref{thm:hLII}),
  \item $\z = 0$ and $xs > x$  (Theorem \ref{thm:hL}).
  \end{enumerate}
(It will also be clear in the proof of (1) that hard Lefschetz fails
in the missing case $\z = 0$ and $xs < x$.)

\begin{remark}
We warn the reader that
the proof in case (1) is comparatively straightforward, and has little in common
with the proofs of cases (2) and (3). On the other hand, the proofs of cases (2) and (3) (which use positivity considerations in
a crucial way) are
similar, with (3) being more involved. The reader is encouraged to view
the proof in case (2) as a warm-up for (3). 
\end{remark}

\begin{thm}[Hard Lefschetz for $\z > 0$, $xs < x$]
\label{thm:hLI} Suppose $\z > 0$ and $xs < x$. If $hL(x)$ holds,
then so does $hL(x,s)_\z$.
\end{thm}

\begin{proof}
  The basic idea is as follows: because $xs < x$ we have $B_xB_s \cong
  B_x(1) \oplus B_x(-1)$. We will fix such an isomorphism and see that the operator 
  $L_\z$ on $\overline{B_xB_s} = \overline{B_x(1)} \oplus \overline{B_x(-1)}$ has the form
\begin{equation} \label{eq:Leasy} L_\z = 
\left (
\begin{matrix}
\rho \cdot (-)  & 0 \\
\z \rho(\a_s^\vee)  &  \rho \cdot (-) \end{matrix} \right )
\end{equation}
where $\rho \cdot (-)$ is the Lefschetz operator on $\overline{B_x}$ given by left
multiplication by $\rho$, and $\z \rho(\alpha_s^\vee)$ denotes a
scalar multiple of the identity, viewed as a degree two map
$\overline{B_x(1)} \to \overline{B_x(-1)}$. Because $\rho \cdot (- )$
satisfies the hard Lefschetz theorem on $\overline{B_x}$, we can
complete the $\rho$-action to an action of $\sl_2(\RM) = \RM f \oplus
\RM h \oplus \RM e$ such that $e = \rho \cdot ( - )$ and $h b = k b$
for all $b \in (\overline{B_x})^k$. In this case, after rescaling
(under the assumption that $\z \ne 0$) the
above matrix describes the action of $e$ on the tensor product of
$\overline{B_x}$ with the standard 2-dimensional representation of
$\sl_2(\RM)$. Hence $e = L_\z$ satisfies the hard Lefschetz theorem as
claimed.

It remains to show that $L_\z$ has the form given in
\eqref{eq:Leasy}. By assumption $xs < x$ and hence, by \cite [Theorem
1.4]{W}, we can find an $(R, R^s)$-bimodule $B_{\overline{x}}$ such
that $B_{\overline{x}} \otimes_{R^s} R \cong B_x$. We conclude that
any choice of isomorphism $R \cong R^s \oplus R^s(-2)$ of graded
$R^s$-modules yields an isomorphism
\begin{equation} \label{eq:singularsplit}
B_xB_s \cong B_{\overline{x}} \otimes_{R^s} R \otimes_{R^s} R(1) \cong
B_x(1) \oplus B_x(-1).
\end{equation}
Now we fix such an isomorphism. Consider the maps $\iota_1, \iota_2 :
R^s \to R$, where $\iota_1$ is the inclusion, and $\iota_2(r) =
\frac{1}{2}\alpha_s\iota_1(r)$. Let $\pi_1, \pi_2 : R \to R^s$ be
given by
\[
\pi_1(r) = \frac{1}{2}(r + sr) \quad \text{and} \quad \pi_2(r) = \partial_s(r).
\]
Then $\pi_a \circ \iota_b = \delta_{ab} $ for $a, b \in \{ 1, 2 \}$ and so these maps give
the inclusions and projections in an $R^s$-bimodule isomorphism $R \cong
R^s \oplus R^s(-2)$. Tensoring these isomorphisms with the identity on
both sides yields the inclusion and projection maps fixing an
isomorphism as in \eqref{eq:singularsplit}.

With respect to this fixed isomorphism a straightforward calculation yields
that $L_\z$ is given by the matrix
\[
\left (
\begin{matrix}
\rho \cdot (- ) + \z(-) \cdot \pi_1(\rho)  & \frac{1}{2} \z (-) \cdot 
\pi_1(\alpha\rho) \\
\z\rho(\a_s^\vee)(-) & \rho \cdot (- ) + \frac{1}{2} \z (-)
\cdot \partial_s(\a \rho)
\end{matrix}
\right ) .
\]
Passing to $\overline{B_xB_s}$ the operator of right multiplication by
a polynomial of positive degree becomes zero, and the above matrix
reduces to \eqref{eq:Leasy}. This completes the proof.
\end{proof}

\begin{thm}[Hard Lefschetz for $\z > 0$, $xs > x$]
\label{thm:hLII} Suppose $\z > 0$ and $xs > x$. Assume:
\begin{enumerate}
\item $S(\le\!x)$ holds;
\item $HR(z,t)$ holds for all $(z,t) \in W \times S$ such that $z < x$
  and $zt > t$;
\item $HR(<\!x,s)_\z$ holds;
\item $HR(x)$ holds.
\end{enumerate}
Then $hL(x,s)_\z$ holds.
\end{thm}


\begin{proof}
  Write $\un{x} = s_1s_2 \cdots s_m$ and set 
\begin{gather*}
\g_i := (s_{i-1} \cdots s_1 \rho)(\a_{s_i}^\vee) \quad
\text{for $1 \le i \le m$ and} \quad 
\g_{m+1} := ( x^{-1} \rho)(\a_s^\vee) + \z \rho (\a_s^\vee).
\end{gather*}
The scalars $\g_1, \dots, \g_m$ are all positive because $\un{x}$ is
reduced, and $\g_{m+1}$ is positive because $xs > x$, see
\eqref{eq:pos}.
As in \S\ref{sec:fact-lefsch-oper} we use the
  tuple $\gamma_{\le m} = (\gamma_i)_{i=1}^m$ (resp. 
  $\gamma = (\gamma_i)_{i=1}^{m+1}$) to
define a rescaled intersection forms $\langle -, - \rangle^{\g_{\le
    m}}$ (resp. $\langle -, - \rangle^\g$) on $\bigoplus
BS(\un{x}_{\hat{i}})$ (resp. $\bigoplus
BS((\un{x}s)_{\widehat{i}})$). 
By a slight variant of \S\ref{sec:fact-lefsch-oper} we have the relation
\begin{equation} \label{eq:phiL}
\langle b, L_\z b'
\rangle_{\ov{BS(\un{x}s)}} = \langle \phi(b), \phi(b') \rangle^\g_\RM \quad \text{for all $b, b' \in \ov{BS(\un{x}s)}$}
\end{equation}
where $\phi$ is the first differential in the complex $F_{s_1} \cdots
F_{s_m}F_s$ and $\langle -, - \rangle^\g_\RM$ denotes the form
on $\bigoplus \overline{BS((\un{x}s)_{\widehat{i}})}$ induced by
$\langle -, - \rangle^\g$.

Now fix a minimal complex $F_x \summand F_{s_1}F_{s_2} \cdots
F_{s_m}$. Because we assume $S(\le\!x)$, Theorem \ref{thm:rouquier
  linear} allows us to conclude that ${}^0F_x = B_x$ and that ${}^kF_x$ is concentrated
in perverse degree $-k$. Because we assume $S(\le\!x)$
  and $HR(z,t)$ for all
$z < x$ with $zt > z$ we may apply Proposition 
\ref{prop:signs} to find an embedding $F_x
\summand F_{s_1} \dots F_{s_m}$ so that ${}^1F_x(-1) \summand \bigoplus
\overline{BS(\un{x}_{\widehat{i}})}$ satisfies the Hodge-Riemann
bilinear relations with respect to the form induced by $\langle -, - \rangle^{\g_{\le
    m}}$.

The first two terms of $F_xF_s$ have the form
\[
B_xB_s \stackrel{\phi}{\longto} {}^1 F_xB_s \oplus B_x(1).
\]
We use this decomposition
to write $\phi = (d_1, d_2)$ for maps $d_1 : B_xB_s \to {}^1F_xB_s(-1)$
and $d_2 : B_xB_s \to B_x(1)$. It is straightforward to verify that 
$d_1$ commutes with $L_\z$, and that for $d_2$ we have
\[
d_2(L_{\z} b) = \rho \cdot d_2(b) + d_2(b) \cdot \z \rho
\]
 for all $b \in B_xB_s$. Hence, if we denote by $L$ the operator on
 $\overline{{}^1 F_xB_s} \oplus \overline{B_x}(1)$ given by $L_\z$ on the first summand and
 $\rho \cdot (- )$ on the second, then we have
 \begin{equation}
   \label{eq:dphicommute}
   \overline{\phi}(L_\z b) = L \overline{\phi}(b) \quad \text{for all $b \in
\overline{B_xB_s}$}
 \end{equation}
where $\overline{\phi}$ denotes the induced map
$\ov{\phi}: \overline{B_xB_s} \to \overline{{}^1 F_xB_s}
\oplus \overline{B_x}(1)$. Moreover:
  \begin{enumerate}
  \item $\overline{\phi}$ is injective in degrees $\le \ell(x)$ (by
    Lemma \ref{lem:Rouqcohom}).
  \item $\langle b, L_\z b' \rangle_{\overline{B_xB_s}} = \langle
    \overline{\phi}(b), \overline{\phi}(b') \rangle^\g_{\RM}$ for all
    $b, b' \in \overline{B_xB_s}$ (by \eqref{eq:phiL}).
\item $\overline{{}^1 F_xB_s(-1)}
\oplus \overline{B_x}$ satisfies the Hodge-Riemann bilinear relations
with respect to the Lefschetz operator $L$ and the form
$\langle -, - \rangle^\g_\RM$. (The decomposition $\overline{{}^1 F_xB_s(-1)}
\oplus \overline{B_x}$ is orthogonal. For $\overline{B_x}$ the
Hodge-Riemann bilinear relations hold by assumption. For $\overline{{}^1
  F_xB_s(-1)}$ the Hodge-Riemann relations hold by Proposition
\ref{prop:sumHR=>HR} and our assumption $HR(y,s)_\z$ for all $y < x$.)
  \end{enumerate}
Now we can apply   Lemma \ref{lem:
  weak lefschetz substitute} to conclude that $L_\z^k :
(\overline{B_xB_s})^{-k} \to (\overline{B_xB_s})^{k}$ is injective for
all $k \ge 0$. Finally, $\overline{B_xB_s}$ is self-dual as a graded
vector space, and hence has symmetric Betti numbers. Hence $L_\z$
satisfies the hard Lefschetz theorem on $\overline{B_xB_s}$ as
claimed.
\end{proof}

\begin{thm}[Hard Lefschetz for $\z = 0$, $xs > x$]
 \label{thm:hL}
Assume:
\begin{enumerate}
\item $S(\le \! x)$ holds;
\item $HR(y,t)$ holds for
  all $(y,t) \in W \times S$ such that $y < x$ and $yt >
  y$;
\item $HR(x)$ holds;
\item $hL(z)$ holds for all $z < xs$.
\end{enumerate}
 Then $hL(x,s)$ holds.
\end{thm}

\begin{proof} Write $\un{x} = s_1s_2 \cdots s_m$ for $x$ and set 
\begin{gather*}
\g_i := ( s_{i-1} \cdots s_1 \rho)(\a_{s_i}^\vee)  \quad
\text{for $1 \le i \le m$ and} \quad 
\g_{m+1} := ( x^{-1} \rho)(\a_s^\vee).
\end{gather*}
By \eqref{eq:pos}, $\gamma_1, \dots, \gamma_{m+1}$ are positive.
 As in \S\ref{sec:fact-lefsch-oper} we use the
  tuple $\gamma_{\le m} = (\gamma_i)_{i=1}^m$ (resp. 
  $\gamma = (\gamma_i)_{i=1}^{m+1}$) to
define a rescaled intersection forms $\langle -, - \rangle^{\g_{\le
    m}}$ (resp. $\langle -, - \rangle^\g$) on $\bigoplus
BS(\un{x}_{\hat{i}})$ (resp. $\bigoplus
BS((\un{x}s)_{\widehat{i}})$). By
\S\ref{sec:fact-lefsch-oper} we have the relation
\begin{equation} \label{eq:phiL2}
\langle b, \rho \cdot b'
\rangle_{\ov{BS(\un{x}s)}} = \langle \phi(b), \phi(b') \rangle_\RM^\g \quad \text{for all $b, b' \in \ov{BS(\un{x}s)}$}
\end{equation}
where $\phi$ is the first differential in the complex $F_{s_1} \cdots
F_{s_m}F_s$ and $\langle -, - \rangle^\g_\RM$ denotes the form
on $\bigoplus \overline{BS((\un{x}s)_{\widehat{i}})}$ induced by
$\langle -, - \rangle^\g$.

We now choose a minimal subcomplex $F_x \summand F_{s_1}F_{s_2} \cdots
F_{s_m}$. We know that ${}^0F_x = B_x$, that ${}^kF_x$ is concentrated
in perverse degree $-k$ by Theorem \ref{thm:rouquier linear},
and that we can choose our embedding such that ${}^kF_x(-k) \summand \bigoplus BS(\un{x})$
satisfies the Hodge-Riemann bilinear relations by Proposition
\ref{prop:signs}. As in the proof of Proposition \ref{prop:signs} let
us decompose
\[
{}^1F_x(-1) = B^{\uparrow} \oplus B^{\downarrow}
\]
so that $B^{\uparrow}B_s$ is perverse and $\HC^0(B^{\downarrow}B_s)
= 0$. This decomposition is orthogonal with respect to
$\langle -, - \rangle^{\g_{\le m}}$ because
\[ \Hom(B^\uparrow, \DM B^\downarrow)
= \Hom(B^\downarrow, \DM B^\uparrow)
= 0.\]

The first two terms of $F_xF_s$ have the form
\[
B_xB_s \to B_x(1) \oplus B^{\uparrow}B_s(1)
\oplus B^{\downarrow}B_s(1).
\]
We claim that this decomposition of ${}^1(F_xF_s)$ is orthogonal with respect to
$\langle -, - \rangle^\g$. Indeed, under the inclusion of ${}^1(F_xF_s)
\summand \bigoplus BS((\underline{x}s)_{\widehat{i}})$ we have $B_x(1)
\summand BS(\un{x})$ and $B^{\uparrow}B_s(1)
\oplus B^{\downarrow}B_s(1) \summand \bigoplus BS(\un{x}_{\hat{i}}s)$. Hence $B_x(1)$ is orthogonal to $B^\uparrow
B_s(1) \oplus B^{\downarrow}B_s(1)$. The form on $B^{\uparrow}B_s(1)
\oplus B^{\downarrow}B_s(1) = (B^\uparrow \oplus B^\downarrow)B_s$
coincides with the induced form from $\langle -, -\rangle^{\gamma_{\le
    m}}$ on $B^\uparrow \oplus B^\downarrow$ (see \S\!\ref{sec:invariant
  forms}). The claimed
orthogonality for the decomposition of ${}^1(F_xF_s)$ now
follows from the orthogonality of $B^\uparrow$ and $B^\downarrow$
under $\langle -, -\rangle^{\gamma_{\le m}}$.

We also know that $F_xF_s \in {}^pK^b(\BC)^{\ge
  0}$ by Corollary \ref{lem:p>0}  and hence the restriction of the second differential to
$\tau_{\le -2}({}^1(F_xF_s)) = \tau_{\le -2}(B^{\downarrow} B_s(1))$ is
a split injection. Canceling this contractible direct summand we
obtain a summand of $F_xF_s$ such that the inclusion is a homotopy
equivalence. Observing that 
$\tau_{\ge -1}(B^\downarrow B_s(1)) = \tau_{\ge 0}(B^{\downarrow}B_s(1)) \cong
B^{\downarrow}$ we see that the first two terms of this summand have the
form
\[
B_xB_s \stackrel{d}{\longrightarrow} B_x(1) \oplus B^{\uparrow}B_s(1)
\oplus B^{\downarrow}.
\]
We use this decomposition to write $d =
(d_1, d_2, d_3)$ for maps $d_1 : B_xB_s \to B_x(1)$,
$d_2 :  B_xB_s \to B^{\uparrow}B_s(1)$ and  $d_3 : B_xB_s \to
B^{\downarrow}$. Consider the induced map
\[
\ov{B_xB_s}  \stackrel{\ov{d}}{\longrightarrow} \ov{B_x}(1) \oplus \ov{B^{\uparrow}B_s}(1)
\oplus \ov{B^{\downarrow}}
\]
with components $\ov{d_1}$, $\ov{d_2}$ and $\ov{d_3}$. By Lemma
\ref{lem:Rouqcohom}, $\overline{d}$ is
injective in degrees $\le \ell(x)$.

Now fix $0 \ne b \in (\ov{B_xB_s})^{-k}$ for some $k \ge 0$.
Because $\ov{B_xB_s}$ has symmetric Betti numbers, to prove the theorem
it is enough to show that $\rho^k \cdot b \ne 0$. Because $\ov{d}(b) \ne
0$, the theorem follows from the following two claims:

\emph{Claim 1:} If $\overline{d_3}(b) \ne 0$ then $\rho^k(b) \ne 0$.

Each indecomposable summand of $B^\downarrow$ is of the form $B_z$
with $z < xs$ and $zs < z$. For such $z$ left multiplication by
$\rho$ on $\overline{B_z}$ satisfies the hard Lefschetz theorem by
assumption. Hence left multiplication by $\rho$ satisfies the hard Lefschetz
theorem on $\overline{B^\downarrow}$. Now $\overline{d_3}$ commutes with left
multiplication by $\rho$. Hence $0 \ne \rho^k(\overline{d_3}(b)) =
\overline{d_3}(\rho^k(b))$ and the claim follows.

\emph{Claim 2:} If $\overline{d_3}(b) = 0$ then $\rho^k(b) \ne 0$.

Consider $V := \Ker(\overline{d_3}) \subset \ov{B_xB_s}$, and $W = \ov{B_x}
\oplus \ov{B^{\uparrow} B_s}$. By restricting $\langle -,-
\rangle_{\overline{B_xB_s}}$ to $V$ and $\langle -,- \rangle_\RM^\g$ to
$W$, we obtain graded forms on these spaces. The operator given by
left multiplication by $\rho$ is a Lefschetz operator on both
spaces. Write $\phi_V$ for the restriction of 
$\overline{d}$ to $V$, viewed as a map $V \to W(1)$. Then:
\begin{enumerate}
\item $\phi_V(\rho b) = \rho(\phi_V(b))$ for all $b \in V$;
\item $\phi_V$ is injective in degrees $\le -1$ (or even $\le \ell(x)$);
\item $\langle b, L_\z b' \rangle_V = \langle \phi_V(b), \phi_V(b')
  \rangle_W$ for all $b, b' \in V$ (by \eqref{eq:phiL});
\item $W$
satisfies the Hodge-Riemann bilinear relations. (For $\overline{B_x}$ this holds
by assumption. For $\ov{B^{\uparrow} B_s}$ this holds because every
indecomposable summand of $B^\uparrow$ is of the form $B_z$ with $zs >
z$. Hence the Hodge-Riemann bilinear relations hold for
$\ov{B^{\uparrow} B_s}$ by our assumption (2) in the 
statement of the theorem, combined with Proposition
\ref{prop:sumHR=>HR} and the fact that $\overline{B^\uparrow}$
satisfies the Hodge-Riemann bilinear relations.)
\end{enumerate}
We now apply Lemma \ref{lem: weak lefschetz substitute} to
conclude that $\rho^k : V^{-k} \to V^k$ is injective.
\end{proof}

\end{document}